\documentclass[11pt]{article}
\usepackage[margin=1in]{geometry}
\usepackage{amsmath,amsthm,amsfonts,amssymb}
\usepackage{graphicx,enumerate}
\usepackage{subcaption}
\graphicspath{ {./figures/} }
\usepackage[title]{appendix}
\usepackage[
            CJKbookmarks=true,
            bookmarksnumbered=true,
            bookmarksopen=true,
            linktocpage=true,
            colorlinks=true,
            citecolor=red,
            linkcolor=blue,
            anchorcolor=red,
            urlcolor=blue
            ]{hyperref}

\newcommand{\eps}{\varepsilon}
\renewcommand{\P}{\mathbb{P}}

\usepackage{tikz}
\usepackage{tikz-qtree}

\usepackage{tikz-cd}

\usetikzlibrary{calc,arrows,cd}
\usetikzlibrary{decorations.pathreplacing,decorations.markings}

\tikzstyle{dot}=[circle,fill,black,inner sep=1pt]

\tikzset{
  on each segment/.style={
    decorate,
    decoration={
      show path construction,
      moveto code={},
      lineto code={
        \path [#1]
        (\tikzinputsegmentfirst) -- (\tikzinputsegmentlast);
      },
      curveto code={
        \path [#1] (\tikzinputsegmentfirst)
        .. controls
        (\tikzinputsegmentsupporta) and (\tikzinputsegmentsupportb)
        ..
        (\tikzinputsegmentlast);
      },
      closepath code={
        \path [#1]
        (\tikzinputsegmentfirst) -- (\tikzinputsegmentlast);
      },
    },
  },
  mid arrow/.style={postaction={decorate,decoration={
        markings,
        mark=at position .5 with {\arrow[#1]{stealth}}
      }}},
  early arrow/.style={postaction={decorate,decoration={
        markings,
        mark=at position .2 with {\arrow[#1]{stealth}}
      }}},
}

\tikzstyle{int}=[draw, fill=blue!15, minimum size=2em]
\tikzstyle{init} = [pin edge={to-,thin,black}]

\def\alternatecolorred{%
    \pgfkeysalso{red}%
    \global\let\alternatecolor\alternatecolorblue 
}
\def\alternatecolorblue{%
    \pgfkeysalso{blue}%
    \global\let\alternatecolor\alternatecolorred 
}

\newcommand{\altred}{\let\alternatecolor\alternatecolorred 
\tikzset{every edge/.append code = {%
    \global\let\currenttarget\tikztotarget 
    \pgfkeysalso{append after command={(\currenttarget)}}
			\alternatecolor
}}
}
\newcommand{\altblue}{\let\alternatecolor\alternatecolorblue 
\tikzset{every edge/.append code = {%
    \global\let\currenttarget\tikztotarget 
    \pgfkeysalso{append after command={(\currenttarget)}}
			\alternatecolor
}}
}

\tikzstyle{vertexdot}=[circle, draw, fill=black, minimum size=3,inner sep=0pt]

\newtheorem{theorem}{Theorem}
\newtheorem{lemma}{Lemma}[section]

\theoremstyle{definition}

\newtheorem{definition}{Definition}

\usepackage{color}
\usepackage{algorithm}
\usepackage{algorithmic}

\usepackage{xspace,prettyref}

\newcommand{\Prob}{\mathbb{P}}

\newcommand{\toprob}{\xrightarrow{\Prob}}

\newcommand{\Var}{\mathsf{Var}}

\newrefformat{eq}{(\ref{#1})}
\newrefformat{chap}{Chapter~\ref{#1}}
\newrefformat{sec}{Section~\ref{#1}}
\newrefformat{alg}{Algorithm~\ref{#1}}
\newrefformat{fig}{Fig.~\ref{#1}}
\newrefformat{tab}{Table~\ref{#1}}
\newrefformat{rmk}{Remark~\ref{#1}}
\newrefformat{clm}{Claim~\ref{#1}}
\newrefformat{def}{Definition~\ref{#1}}
\newrefformat{cor}{Corollary~\ref{#1}}
\newrefformat{lmm}{Lemma~\ref{#1}}
\newrefformat{prop}{Proposition~\ref{#1}}
\newrefformat{app}{Appendix~\ref{#1}}
\newrefformat{hyp}{Hypothesis~\ref{#1}}
\newrefformat{thm}{Theorem~\ref{#1}}
\newrefformat{ass}{Assumption~\ref{#1}}
\newrefformat{conj}{Conjecture~\ref{#1}}

\newcommand{\pth}[1]{\left( #1 \right)}
\newcommand{\qth}[1]{\left[ #1 \right]}
\newcommand{\sth}[1]{\left\{ #1 \right\}}

\newcommand{\iprod}[2]{\left \langle #1, #2 \right\rangle}
\newcommand{\Iprod}[2]{\langle #1, #2 \rangle}

\newcommand{\diag}{\mathsf{diag}}

\newcommand{\sfT}{{\mathsf{T}}}

\newcommand{\calA}{{\mathcal{A}}}
\newcommand{\calB}{{\mathcal{B}}}

\newcommand{\calE}{{\mathcal{E}}}
\newcommand{\calF}{{\mathcal{F}}}

\newcommand{\calI}{{\mathcal{I}}}

\newcommand{\calN}{{\mathcal{N}}}

\newcommand{\calS}{{\mathcal{S}}}
\newcommand{\calT}{{\mathcal{T}}}

\DeclareMathAlphabet{\varmathbb}{U}{bbold}{m}{n}

\newcommand{\EE}{\mathbb{E}}

\renewcommand{\d}{{\rm d}}

\renewcommand{\hat}{\widehat}

\usepackage{bbm}

\newcommand{\R}{\mathbb{R}}
\newcommand{\C}{\mathbb{C}}

\newcommand{\1}{\mathbbm{1}}

\newcommand{\ii}{\mathrm{i}}

\newcommand{\ER}{Erd\H{o}s-R\'{e}nyi\xspace}

\newcommand{\bfX}{\mathbf{X}}
\newcommand{\bfY}{\mathbf{Y}}
\newcommand{\bfH}{\mathbf{H}}

\newcommand{\bfR}{\mathbf{R}}

\newcommand{\bfU}{\mathbf{U}}
\newcommand{\bfV}{\mathbf{V}}

\newcommand{\tbfX}{\widetilde{\bfX}}
\newcommand{\tbfY}{\widetilde{\bfY}}

\newcommand{\bfE}{\mathbf{E}}
\newcommand{\tbfE}{\widetilde{\mathbf{E}}}

\newcommand{\bfP}{\mathbf{P}}
\newcommand{\bfQ}{\mathbf{Q}}
\newcommand{\Pt}{\widetilde{\bfP}^{[t]}}
\newcommand{\Qt}{\widetilde{\bfQ}^{[t]}}
\newcommand{\Ptt}{\widetilde{\bfP}^{[t+1]}}
\newcommand{\Qtt}{\widetilde{\bfQ}^{[t+1]}}

\newcommand{\frakq}{\mathfrak{q}}
\newcommand{\frakp}{\mathfrak{p}}
\newcommand{\frakw}{\mathfrak{w}}
\newcommand{\frakf}{\mathfrak{f}}
\newcommand{\frakg}{\mathfrak{g}}

\newcommand{\bfI}{\mathbf{I}}
\newcommand{\bfG}{\mathbf{G}}

\newcommand{\hatX}{\hat{\bfX}}
\newcommand{\hatXk}{\hat{\bfX}^{[k]}}
\newcommand{\hatu}{\hat{\u}}
\newcommand{\hatv}{\hat{\v}}
\newcommand{\hatuk}{\hat{\u}^{[k]}}
\newcommand{\hatvk}{\hat{\v}^{[k]}}
\newcommand{\hatpsi}{\hat{\psi}}

\renewcommand{\u}{\mathbf{u}}
\renewcommand{\v}{\mathbf{v}}
\newcommand{\w}{\mathbf{w}}

\newcommand{\Xk}{\bfX^{[k]}}
\newcommand{\Yk}{\bfY^{[k]}}

\newcommand{\Hk}{\bfH^{[k]}}

\newcommand{\Rk}{\bfR^{[k]}}

\newcommand{\Rt}{\bfR^{[t]}}
\newcommand{\Rtt}{\bfR^{[t+1]}}
\newcommand{\Xt}{\bfX^{[t]}}

\newcommand{\tXt}{\tbfX^{[t]}}
\newcommand{\tXtt}{\tbfX^{[t+1]}}

\newcommand{\Dk}{\Delta^{[k]}}

\newcommand{\tXk}{\tbfX^{[k]}}
\newcommand{\tYk}{\tbfY^{[k]}}

\newcommand{\Lk}{\lambda^{[k]}}
\newcommand{\Sigk}{\sigma^{[k]}}
\newcommand{\Muk}{\mu^{[k]}}

\newcommand{\uk}{\u^{[k]}}
\newcommand{\vk}{\v^{[k]}}
\newcommand{\wk}{\w^{[k]}}

\newcommand{\f}{\mathbf{f}}
\newcommand{\g}{\mathbf{g}}
\newcommand{\h}{\mathbf{h}}
\newcommand{\fk}{\mathbf{f}^{[k]}}
\newcommand{\gk}{\mathbf{g}^{[k]}}
\newcommand{\hk}{\mathbf{h}^{[k]}}

\renewcommand{\a}{\mathbf{a}}
\renewcommand{\b}{\mathbf{b}}

\newcommand{\Pis}{\mathbf{\Pi}_{\mathrm{s}}}
\newcommand{\Pif}{\mathbf{\Pi}_{\mathrm{f}}}

\newcommand{\MP}{\mathsf{MP}}
\newcommand{\im}{\text{\rm Im}{\,}}

\newcommand{\abs}[1]{\left| {#1} \right|}

\begin{document}

\pgfdeclarelayer{background}
\pgfdeclarelayer{foreground}
\pgfsetlayers{background,main,foreground}

\title{
Resampling Sensitivity of High-Dimensional PCA
}

\author{
Haoyu Wang \thanks{Department of Mathematics, Yale University, New Haven, CT 06511, USA,  \texttt{haoyu.wang@yale.edu}}
}

\date{\today}

\maketitle

\begin{abstract}

The study of stability and sensitivity of statistical methods or algorithms with respect to their data is an important problem in machine learning and statistics. The performance of the algorithm under resampling of the data is a fundamental way to measure its stability and is closely related to generalization or privacy of the algorithm.
In this paper, we study the resampling sensitivity for the principal component analysis (PCA). Given an $ n \times p $ random matrix $ \bfX $, let $ \Xk $ be the matrix obtained from $ \bfX $ by resampling $ k $ randomly chosen entries of $ \bfX $. Let $ \v $ and $ \vk $ denote the principal components of $ \bfX $ and $ \Xk $.
In the proportional growth regime $ p/n \to \xi \in (0,1] $, we establish the sharp threshold for the sensitivity/stability transition of PCA. When $ k \gg n^{5/3} $, the principal components $ \v $ and $ \vk $ are asymptotically orthogonal. On the other hand, when $ k \ll n^{5/3} $, the principal components $ \v $ and $ \vk $ are asymptotically colinear. In words, we show that PCA is sensitive to the input data in the sense that resampling even a negligible portion of the input may completely change the output.

\end{abstract}










\newpage

\tableofcontents

\newpage

\section{Introduction}

The study of stability and sensitivity of statistical methods and algorithms with respect to the input data is an important task in machine learning and statistics \cite{bousquet2002stability,elisseeff2005stability,mukherjee2006learning,hardt2016train,deng2021toward}. The notion of stability for algorithms is also closely related to differential privacy \cite{dwork2014algorithmic} and generalization error \cite{kutin2002almost}. To measure algorithmic stability, one fundamental question is to study the performance of the algorithm under resampling of its input data  \cite{barber2021predictive,kim2021black}.
Originating from the analysis of Boolean functions, resampling sensitivity (also called noise sensitivity) is an important concept in theoretical computer science, which refers to the phenomenon that resampling a small portion of the random input data may lead to decorrelation of the output. Such a remarkable phenomenon was first studied in the pioneering work of Benjamini, Kalai and Schramm \cite{Schramm}, and we refer to the monograph \cite{garban2014noise} for a systematic discussion on this topic.




In this work, we study the resampling sensitivity of principal component analysis (PCA). As one of the most commonly used statistical methods, PCA is widely applied for dimension reduction, feature extraction, etc \cite{johnstone2007icm,dieng2011application}. It is also used in other fields such as economics \cite{vyas2006constructing}, finance \cite{ait2017using}, genetics \cite{ringner2008principal}, and so on. The performance of PCA under the additive or multiplicative independent perturbation of the data matrix has been well studied (see e.g. \cite{baik2005phase,baik2006eigenvalues,paul2007asymptotics,benaych2011eigenvalues,candes2011robust,fan2018eigenvector}). 

However, how resampling of the data matrix affects the outcome remains unclear. In this paper, we address this problem for the first time. Here, we emphasize that the resampling of the input data may not have any structure, and the specific resampling procedure is given in the next subsection. In our main results, we show that PCA is resampling sensitive, in the sense that, above certain threshold, resampling even a negligible portion of the data may make the resulted principal component completely change (i.e. become orthogonal to the original direction).

Compared with previous work that mainly focused on PCA with additive or multiplicative independent noise, our setting is very different. In our model, if writing the resampling effect as an additive or multiplicative perturbation, then this noise is not independent of the signal and does not possess any special structure. In contrast, in previous work, sometimes low-rank assumptions on the structure of the matrix or the noise, or some kind of incoherence conditions were imposed. In our work, we have almost no assumption on the data other than a sub-exponential decay condition. Moreover, we highlight that our results have universality. In particular, we do not need to know the specific distribution of the data and we do not require the data is i.i.d sampled.





\subsection{Model and main results}

Let $ \bfX=(\bfX_{ij}) $ be an $ n \times p $ data matrix with independent real valued entries with mean 0 and variance $ p^{-1} $,
\begin{equation}\label{e.Assumption1}
\bfX_{ij} = p^{-1/2}x_{ij},\ \ \ \EE [x_{ij}]=0,\ \ \ \EE[ x_{ij}^2 ]=1.
\end{equation}
Note that we do not require the i.i.d. condition for the data. Furthermore, we assume the entries $ x_{ij} $ have a sub-exponential decay, that is, there exists a constant $ \theta>0 $ such that for $ u>1 $,
\begin{equation}\label{e.Assumption2}
\P (|x_{ij}| > u) \leq \theta^{-1} \exp (-u^\theta).
\end{equation}
This sub-exponential decay assumption is mainly for convenience, and other conditions such as the finiteness of a sufficiently high moment would be enough.

Motivated by high-dimensional statistics, we will work in the proportional growth regime $ n \asymp p $. Throughout this paper, to avoid trivial eigenvalues, we will be working in the regime
$$ 
\lim_{n \to \infty} p/n = \xi \in (0,1)\ \ \mbox{or}\ \ p/n \equiv 1.
$$
In the case $ \lim p/n=1 $, our assymption $ p/n \equiv 1 $ is due to technical reasons in random matrix theory. Specifically, the proof relies on the delocalization of eigenvectors in the whole spectrum. As one of the major open problems in random matrix theory, delocalization of eigenvectors near the lower spectral edge is not known in the general case with just $ \lim p/n=1 $. The strictly square assumption $ p \equiv n $ can be slightly relaxed to $ |n-p| = p^{o(1)} $ (see e.g. \cite{wang2022optimal}), but we do not pursue such an extension for simplicity.

The sample covariance matrix corresponding to data matrix $ \bfX $ is defined by $ \bfH := \bfX^\top \bfX $. We order the eigenvalues of $ \bfH $ as $ \lambda_1 \geq \cdots \geq \lambda_p $, and use $ \v_i \in \R^p $ to denote the unit eigenvector corresponding to the eigenvalue $ \lambda_i $. If the context is clear, we just use $ \lambda:=\lambda_1 $ and $ \v:=\v_1 $ to denote the largest eigenvalue and the top eigenvector. 
We also consider the eigenvalues and eigenvectors of the Gram matrix $ \widehat{\bfH}:=\bfX \bfX^\top $. Note that $ \widehat{\bfH} $ and $ \bfH $ have the same non-trivial eigenvalues, and the spectrum of $ \widehat{\bfH} $ is given by $ \{\lambda_i\}_{i=1}^n $ with $ \lambda_{p+1}=\cdots=\lambda_n=0 $. We denote the unit eigenvectors of $ \widehat{\bfH} $ associated with the eigenvalue $ \lambda_i $ by $ \u_i \in \R^n $.

Let $ \bfU=[\u_1,\cdots,\u_p] \in \R^{n \times p} $ and $ \bfV=[\v_1,\cdots,\v_p] \in \R^{p \times p} $. These eigenvectors may be connected by the singular value decomposition of the data matrix $ \bfX=\bfU \mathbf{\Sigma} \bfV^\top $, where $ \mathbf{\Sigma} := \diag(\sigma_1,\cdots,\sigma_p) $ with $ \sigma_i=\sqrt{\lambda_i} $ corresponds to the singular values. For convenience, we also denote $ \sigma := \sigma_1 $.
And therefore, up to the sign of the eigenvectors, we have
$$ \bfX \v_\alpha = \sqrt{\lambda_\alpha} \u_\alpha,\ \ \bfX^\top \u_\alpha = \sqrt{\lambda_\alpha} \v_\alpha.   $$


We now describe the resampling procedure. For a positive number $ k \leq np $, define the random matrix $ \Xk $ in the following way. Let $ S_k=\{(i_1,\alpha_1),\cdots,(i_k,\alpha_k)\} $ be a set of $ k $ pairs chosen uniformly at random without raplacement from the set of all ordered pairs $ (i,\alpha) $ of indices with $ 1 \leq i \leq n $ and $ 1 \leq \alpha \leq p $. We assume that the set $ S_k $ is independent of the entries of $ \bfX $. The entries of $ \Xk $ are given by
\begin{equation*}
\Xk_{i,\alpha}=
\left\{
\begin{aligned}
& \bfX_{i,\alpha}' & \quad & \mbox{if } (i,\alpha) \in S_k,\\
& \bfX_{i,\alpha} & \quad & \mbox{otherwise},
\end{aligned}
\right.
\end{equation*}
where $ (\bfX_{i,\alpha}')_{1 \leq i \leq n,1 \leq \alpha \leq p} $ are independent random variables, independent of $ \bfX $, and $ \bfX_{i,\alpha}' $ has the same distribution as $ \bfX_{i,\alpha} $. In other words, $ \Xk $ is obtained from $ \bfX $ by resampling $ k $ random entries of the matrix, and therefore $ \Xk $ clearly has the same distribution as $ \bfX $. Let $ \Hk:=(\Xk)^\top \Xk $ the sample covariance matrix corresponding to the resampled data matrix $ \Xk $. Denote the eigenvalues and the corresponding normalized eigenvectors of $ \Hk $ by $ \Lk_1 \geq \cdots \geq \Lk_p  $ and $ \vk_1,\cdots,\vk_p $. When the context is clear, the principal component is just denoted by $ \Lk $ and $ \vk $. Similarly, denote the eigenvector of the matrix $ \widehat{\bfH}^{[k]}:=\Xk (\Xk)^\top $ associated with the eigenvalue $ \Lk_i $ by $ \uk_i $.

With the resampling parameter in two different regimes, we have the following results.

\begin{theorem}[Noise sensitivity under excessive resampling]\label{thm:main1}
Let $ \bfX $ be a random data matrix satisfying \eqref{e.Assumption1} and \eqref{e.Assumption2} and $ \Xk $ be the resampled matrix defined as above. If $ k \gg n^{5/3} $, then the associated principal components are asymptotically orthogonal, i.e.
\begin{equation}\label{e.main1}
\lim_{n \to \infty} \EE \left| \langle \v,\vk \rangle \right| =0,\ \ \mbox{and}\ \ \lim_{n \to \infty} \EE \left| \langle \u,\uk \rangle \right| =0.
\end{equation}
\end{theorem}

\begin{theorem}[Noise sensitivity under moderate resampling]\label{thm:main2}
Let $ \bfX $ be a random data matrix satisfying \eqref{e.Assumption1} and \eqref{e.Assumption2} and $ \Xk $ be the resampled matrix defined as above. For any $ \epsilon_0>0 $,
\begin{equation}\label{e.main2}
\max_{1 \leq k \leq  n^{5/3-\epsilon_0}} \min_{s \in \{-1,1\}} \sqrt{n} \|\v - s\vk \|_\infty \toprob 0,
\end{equation}
where $ \toprob $ means convergence in probability.
In particular, this implies
\begin{equation*}
\lim_{n \to \infty} \EE \qth{ \max_{1 \leq k \leq n^{5/3 - \epsilon_0}} \min_{s \in \{-1,1\}} \|\v - s\vk \|_2 } = 0.
\end{equation*}
The same result also holds for $ \u $ and $ \uk $.
\end{theorem}

These two theorems together state that the critical threshold for the resampling strength is of order $ k \asymp n^{5/3} $. Note that $ n^{5/3} $ compared with the total number of inputs $ np \asymp n^2 $ is negligible. We show that, above the threshold $ n^{5/3} $, resampling even a negligible portion of the data will result in a dramatic change of the resulting principal component, in the sense that the new principal component is asymptotically orthogonal to the old one; while below the threshold, a relatively mild resampling has almost no effect on the corresponding new principal component. If considering the eigenvector overlaps $ |\Iprod{\v}{\vk}| $ and $ |\Iprod{\u}{\uk}| $, these quantities exhibit sharp phase transitions from $ 1 $ to $ 0 $ near the critical threshold $ k \asymp n^{5/3} $.

We remark that the phase transition stated in the above theorems is not restricted to the top eigenvectors $ \v,\vk,\u,\uk $. With essentially the same arguments, we can prove that for any fixed $ m>0 $, the $ m $-th leading eigenvectors $ \v_m,\vk_m $ and $ \u_m,\uk_m $ exhibit the same phase transition at the critical threshold of the same order $ n^{5/3} $.








\subsection{High-Level Proof Scheme}

The high-level idea is the ``superconcentration implies chaos" phenomenon established by Chatterjee \cite{chatterjee2014superconcentration}, which means that small perturbation (beyond certain threshold) of super-concentrated system ( in the sense that it is characterized by some quantity with small variance) will lead to a dramatic change of the system. For random matrix models, the super-concentrated quantity is usually the eigenvalues. Resampling sensitivity for random matrices was first studied for Wigner matrices and sparse \ER graphs in \cite{bordenave2020noise,bordenave2022noise}. 
For the PCA model, the key difference is that the entries of the sample covariance matrix are correlated. Moreover, resampling a single entry of the data will change $ \Theta(n) $ entries in the sample covariance matrix. These two differences will make the proofs more technical and a linearization trick would be important to reduce the interdependency of the matrix entries. In addition, we also need a technical variance formula related to resampling to compute the variance of the top eigenvalue and tools from random matrix theory, in particular the local Marchenko-Pastur law for the resolvent and the delocalization of eigenvectors.

For the sensitivity of the principal components. We show that the inner products $ \Iprod{\v}{\vk} $ and $ \Iprod{\u}{\uk} $ are closely related to the variance of the top eigenvalue of the sample covariance matrix. Using the variance formula for resampling, we have a precise characterization between the inner product of the top eigenvectors and the concentration of the top eigenvalue. The perturbation effect of the resampling is studied via the variational representation of the eigenvalue.

For the stability under moderate resampling, the key idea of our proof is to study the stability of the resolvent. On the one hand, the stability of resolvent implies the stability of the top eigenvalue. On the other hand, the resolvent can be used to approximate certain useful statistics of the top eigenvector. The stability of the resolvent is proved via a Lindeberg exchange strategy. The resampling procedure can be decomposed into a martingale, and the difference between the resolvents can be therefore bounded by martingale concentration inequalities combined with the local Marchenko-Pastur law of the resolvent for a priori estimates.

\section{Sensitivity under Excessive Resampling}

We provide a heuristic argument for deriving the threshold for the sensitivity regime. We consider the derivative of the top eigenvalue as a function of the matrix entries. Then we have the approximation
$$ \lambda^{[1]} - \lambda \approx \v^\top \qth{(\bfX^{[1]})^\top \bfX^{[1]} - \bfX^\top \bfX} \v $$
Note that the matrix in the parenthesis has only $ \Theta(p) $ non-zero entries, and each entry is roughly of size $ O(p^{-1}) $. Also, the eigenvector $ \v $ is delocalized in the sense that $ \v(m) \approx p^{-1/2} $ for all $ m=1,\dots,p $. A central limit theorem yields that approximately we have
$$ \lambda^{[1]} - \lambda \approx \sqrt{p}p^{-1}p^{-1} = p^{-3/2}. $$
By this heuristic argument and central limit theorem, we have
$$ \Lk - \lambda \approx \sqrt{k} p^{-3/2}. $$
Note that from random matrix theory, we know that $ \lambda_1-\lambda_2 $ is of order $ p^{-2/3} $. Therefore, if we have $ \sqrt{k} p^{-3/2} \ll p^{-2/3} $ (this corresponds to $ k \ll n^{5/3} $), then the difference the two top eigenvalues $ \lambda $ and $ \Lk $ is much smaller than the first two eigenvalues $ \lambda_1 $ and $ \lambda_2 $ of the matrix $ \bfX^\top \bfX $. This implies that the perturbation effect on $ \Xk $ is small, and therefore in this case it is plausible to believe that $ \vk $ is just a small perturbation of $ \v $. Thus, for the sensitivity regime, we must have $ k \gg n^{5/3} $.

Our proof is essentially trying to make the above heuristics rigorous. To do this, a key observation is that the inner products $ \Iprod{\v}{\vk} $ and $ \Iprod{\u}{\uk} $ can be related to the variance of the leading eigenvalue.

\subsection{Connection with Variance of Top Eigenvalue}
As mentioned above, the key step in the proof for sensitivity regime is to establish a connection between the inner products of top eigenvalues with the variance of the top eigenvalue. Specifically, we will prove
$$ \EE \qth{ |\Iprod{\v}{\vk}|^2 } \leq C\frac{n^3 \Var(\sigma)}{k} + o(1), $$
where $ C>0 $ is some universal constant and $ \sigma=\sqrt{\lambda} $ is the leading singular value.
A similar result is also true for $ \u $ and $ \uk $.
For more details, we refer to Section \ref{sec:Proof_Sensitivity}.

From random matrix theory, we have $ \Var(\sigma)=O(n^{-4/3}) $. Then, based on this inequality, we derive the threshold $ k \gg n^{5/3} $ for the sensitivity regime.

\section{Stability under Moderate Resampling}

To establish the stability of PCA when the resampling strength is mild, we will utilize tools from random matrix theory and specifically the proof relies on the analysis of the resolvent matrix. Also, to simplify the Gram matrix structure of the sample covariance matrix, when considering the resolvent we use a linearization trick. For any $ z \in \C $ with $ \im z>0 $, the resolvent is defined as
\begin{equation*}
\bfR(z) := \left(
\begin{matrix}
-\bfI_{n} & \bfX\\
\bfX^\top & -z \bfI_{p}
\end{matrix}
\right)^{-1}.    
\end{equation*}
Similarly, we denote the resolvent of $ \Xk $ by $ \Rk(z) $. The key idea for the proofs in the stability regime is that eigenvectors can be approximated by resolvents and the resolvents are stable under moderate resampling.

\subsection{Resolvent Approximation}
To illustrate the usefulness of the resolvent, we show that the entries of the resolvent can be used to approximate the product of entries in the eigenvector. For some small $ \delta>0 $, let $ z_0=\lambda+\ii \eta $ with $ \eta=n^{-2/3-\delta} $. In the regime $ k \leq n^{5/3-\epsilon_0} $ for some $ \epsilon_0>0 $, there exists some small $ c>0 $ such that for all $ \alpha,\beta=1,\dots,p $, we have
$$  \abs{ \eta \im \bfR_{n+\alpha, n+\beta}(z_0) - \v(\alpha) \v(\beta) } \leq n^{-1-c}, $$
and
$$ \abs{ \eta \im \Rk_{n+\alpha, n+\beta}(z_0) - \vk(\alpha) \vk(\beta) } \leq n^{-1-c}. $$
A similar result also holds for $ \u $ and $ \uk $. For more details, we refer to Lemma \ref{lem:Resolvent_Approx}.

\subsection{Stability of the Resolvent}
Since the eigenvector can be approximated by the resolvent, it suffices to show the stability of the resolvent. Consider the regime $ k \leq n^{5/3-\epsilon_0} $ for some $ \epsilon_0>0 $. For some small $ \delta>0 $ and all $ z=E+\ii\eta $ that is close to the upper spectral edge and $ \eta=n^{-2/3-\delta} $, there exists a small constant $ c>0 $ such that the following is true for all $ i,j=1,\dots,n $ and $ \alpha,\beta=1,\dots,p $,
$$ \abs{ \Rk_{ij}(z) - \bfR_{ij}(z) } \leq \frac{1}{n^{1+c} \eta}, $$
and
$$ \abs{ \Rk_{n+\alpha,n+\beta}(z) - \bfR_{n+\alpha,n+\beta}(z) } \leq \frac{1}{n^{1+c} \eta}. $$
This is the main technical part of the whole argument, and its proof relies on the Lindeberg exchange method and a martingale concentration argument. For more details, we refer to Lemma \ref{lem:Stability_Resolvent}.

Combining the stability of the resolvents with the resolvent approximation for eigenvectors, we can conclude that $ \v $ and $ \vk $ must be close (similarly, also for $ \u $ and $ \uk $).

\subsection{Stability of the Top Eigenvalue}
As a byproduct, we derive an stability estimate of the top eigenvalues, which may be of independent interest. For $ k \leq n^{5/3-\epsilon_0} $ with some arbitrary $ \epsilon_0>0 $ and an arbitrary $ \eps>0 $,
$$ |\Lk - \lambda| \leq n^{-2/3-\delta+\eps}, $$
where $ \delta>0 $ is some small constant. Note that the fluctuation of the top eigenvalue around its typical location is of order $ O(n^{-2/3+\eps}) $, this result shows that the top eigenvalues under moderate resampling are indeed non-trivially stable.
For more details, we refer to Lemma \ref{lem:Stability_Eigenvalue}.


\section{Discussions and Applications}\label{sec:Discussion}

We have shown that PCA is sensitive to the input data, in the sense that resampling $ \omega(n^{-2/3}) $ fraction of the data will results in decorrelation between the new output and the old output. We further prove that this threshold is sharp that a moderate resampling below this threshold will have no effect on the outcome.

Moreover, besides demonstrating an exciting phenomenon, our results have broad implications in other related fields. We briefly discuss a few potential extensions and applications that would be worth further exploration.

\subsection{Extensions to Broader PCA Models}


\paragraph{Sparse PCA}
A natural extension is to consider sparse data, and this corresponds to sparse PCA that received a lot of attention in the past decade (see e.g. \cite{moghaddam2005spectral,zou2006sparse,cai2013sparse}). However, establishing the sharp phase transition for sparse PCA lacks several technical ingredients. In particular, for the sensitivity regime, the eigenvalue gap property (i.e. tail estimates for the eigenvalue gap, see Lemma \ref{lem:Gap}) is unknown. Also, to establish the sharp threshold for the stability regime, an improved local law of the resolvent is needed (cf. Lemma \ref{lem:Local_law}). Though we expect these missing parts to be correct, proving them would be beyond the scope of this paper and we leave them to future work. Assuming the eigenvalue gap property and the improved local law, the phase transition can be proved by the same arguments in this paper.

\paragraph{PCA with General Population}
For practical purposes, it would be significant to consider data with a general population covariance profile (see e.g. \cite{nadler2008finite}). The corresponding matrices were studied in \cite{bao2015universality,lee2016tracy}. The only missing ingredient we need is the eigenvalue gap property. Under reasonable assumptions for the population matrix so that the eigenvalue gap property is true, our arguments in this paper will yield the same stability-sensitivity transition.

\subsection{Extensions to Other Statistical Methods}

Within the general PCA framework, one important variant is the kernel PCA, which is closely related to the widely used spectral clustering \cite{ng2001spectral,von2007tutorial,couillet2016kernel}. The corresponding kernel random matrices were studied in \cite{el2010spectrum,cheng2013spectrum}. However, the study of these kernel random matrices are far from being well-understood. In particular, the study of eigenvectors were very limited and the indispensable property of delocalization is still unknown.

It would be interesting to explore whether other statistical method share the same resampling sensitivity phenomenon. Random matrices associated with canonical correlation analysis (CCA) or multivariate analysis of variance (MANOVA) are well studied \cite{han2016tracy,han2018unified,yang2022sample}. In particular, the Tracy-Widom concentration of the top eigenvalue, one of the most important ingredients that we need, has been proved. We anticipate that these models exhibit a similar stability-sensitivity transition as in PCA. 

\subsection{Differential Privacy}
In our paper, we study the stability of the top eigenvalue and the top eigenvector under resampling in terms of bounding the $ \ell_\infty $ distance. Such stability estimates can be regarded as the global sensitivity of PCA performed on neighboring datasets. This measurement is closely related to the analysis of differential privacy \cite{dwork2014algorithmic}. PCA under differential privacy was previous studied in
\cite{blum2005practical,chaudhuri2013near}, etc. Our result revisit the problem of designing a private algorithm for solving the principal component. Here we remark that though the statements in Theorem \ref{thm:main1} and Theorem \ref{thm:main2} are qualitative, a careful examinination of the proof can yield some quantitative estimates. Based on the stability estimates in terms of the $ \ell_\infty $ metric, a simple Laplace mechanism produces a differentially private version of PCA for computing the top eigenvalue or the top eigenvector. However, compared with \cite{chaudhuri2013near}, our results are limited in the sense that their results are non-asymptotic for all sample size $ n $ and data dimension $ p $, while ours are restricted to the proportional growth regime.

Moreover, previous works on differentially private PCA focused on neighbouring datasets that differ by one sample vector. Our result may be seen as a refined notion of privacy, since we can analyze the sensitivity of PCA over two ``neighbouring" datasets with $ k $ different entries for any $ k $.

Meanwhile, the largest eigenvalue of the sample covariance matrix plays an important rule in hypothesis testing. For example, the Roy's largest root test is used in many problems (see e.g. \cite{johnstone2017roy}). Our result may provide useful insights to construct a differentially private test statistic based on the top eigenvalue.

\subsection{Database Alignment}
Database alignment (or in some cases graph matching) refers to the optimization problem in which we are given two datasets and the goal is to find an optimal correspondence between the samples and features that maximally align the data. For datasets $ \bfX,\bfY \in \R^{n \times p} $, we look for permutations $ \pi_{\mathrm{s}} \in \calS_n $ and $ \pi_{\mathrm{f}} \in \calS_p$ to solve the optimization problem
$$ \max_{\pi_{\mathrm{s}},\pi_{\mathrm{f}}} \sum_{i=1}^n \sum_{\alpha=1}^p \bfX_{i\alpha}\bfY_{\pi_{\mathrm{s}}(i) \pi_{\mathrm{f}}(\alpha)}, $$
where $ \calS_n $ and $ \calS_p $ are the sets of all permutations on $ [n] $ and $ [p] $, respectively.
This problem is closely related to the Quadratic Assignment Problem (QAP), which is known to be NP-hard to solve or even approximate. 

The study of the alignment problem for correlated random databases has a long history. The previous work mainly considered matrices that are correlated through some additive perturbation, and some of the general model were studied with a homogeneous correlation (i.e. the correlation between all correponding pairs are the same). See for example \cite{dai2019database,wu2022settling} and many other works.

Our resampling procedure may be regarded as an adversarial corruption of the dataset, which is a different kind of correlation compared with previous work. To our knowledge, this is the first time to consider database alignment with adversarial corruption.
To state the setting of the problem, we have two matrices 
$$ \bfX \in \R^{n \times p},\ \ \  \bfY=\Pis \Xk \Pif^\top $$
where $ \bfX $ is a random matrix satisfying \eqref{e.Assumption1} and \eqref{e.Assumption2}, and $ \Pis $ and $ \Pif $ are permutation matrices of order $ n $ and $ p $ chosen uniformly at random. The goal is to recover the permutations $ \Pis $ and $ \Pif $ based on the observations $ \bfX $ and $ \bfY $. Here, we can think of $ \bfY $ as the unlabeled version of $ \bfX$ with adversarial corruption. By considering the covariance matrices, we have
$$ \mathbf{A}=\bfX \bfX^\top,\ \ \mathbf{B}=\bfY \bfY^\top = \Pis \pth{ \Xk (\Xk)^\top } \Pis^\top, $$
and similarly we consider
$$ \hat{\mathbf{A}} = \bfX^\top \bfX,\ \ \hat{\mathbf{B}} = \bfY^\top \bfY = \Pif \pth{(\Xk)^\top \Xk} \Pif^\top. $$
A natural idea to reconstruct the permutations $ \Pis $ (and $ \Pif) $ is to align the top eigenvectors of the matrices $ \mathbf{A} $ and $ \mathbf{B} $ (and $ \hat{\mathbf{A}} $ and $ \hat{\mathbf{B}} $). See Algorithm \ref{alg:PCA} for details. This spectral method is a natural technique for databse alignment and graph matching. We are interested in under what resampling strength, the PCA-Recovery algorithm can almost perfectly reconstruct the permutations, and under what condition this method completely fail.

\begin{algorithm}[tb]
   \caption{PCA-Recovery}
   \label{alg:PCA}
\begin{algorithmic}
   \STATE {\bfseries Input:} data matrices $\bfX,\bfY \in \R^{n \times p}$
   \STATE {\bfseries Output:} permutation matrices $ \Pis \in \R^{n \times n} $, $ \Pif \in \R^{p \times p} $
   \STATE Compute $ \u $ the unit leading left singular vectors of $ \bfX $
   \STATE Compute $ \v $ the unit leading right singular vectors of $ \bfX $
   \STATE Compute $ \u' $ the unit leading left singular vectors of $ \bfY $
   \STATE Compute $ \v' $ the unit leading right singular vectors of $ \bfY $
   \STATE
   \STATE Compute $ \Pis^{+} $ the permutation aligning $ \u $ and $ \u' $
   \STATE Compute $ \Pis^{-} $ the permutation aligning $ \u $ and $ -\u' $
   \STATE Compute $ \Pif^{+} $ the permutation aligning $ \v $ and $ \v' $
   \STATE Compute $ \Pif^{-} $ the permutation aligning $ \v $ and $ -\v' $
   \STATE 
   
   \IF{$ \Iprod{{\mathbf{A}}}{\Pis^{+} {\mathbf{B}} (\Pis^{+})^\top} \geq \Iprod{{\mathbf{A}}}{\Pis^{-} {\mathbf{B}} (\Pis^{-})^\top} $}
   \STATE $\Pis \leftarrow \Pis^{+} $
   \ELSE
   \STATE $\Pis \leftarrow \Pis^{-} $
   \ENDIF
   \STATE
   \IF{$ \Iprod{\hat{\mathbf{A}}}{\Pif^{+} \hat{\mathbf{B}} (\Pif^{+})^\top} \geq \Iprod{\hat{\mathbf{A}}}{\Pif^{-} \hat{\mathbf{B}} (\Pif^{-})^\top} $}
   \STATE $\Pif \leftarrow \Pif^{+} $
   \ELSE
   \STATE $\Pif \leftarrow \Pif^{-} $
   \ENDIF

\end{algorithmic}
\end{algorithm}

Spectral methods have been studied and applied in many scenarios (see e.g. \cite{fan2020spectral,fan2022spectral} and \cite{ganassali2022spectral}). In particular, in \cite{ganassali2022spectral}, a similar PCA method was studied to match two symmetric Gaussian matrices correlated via additive Gaussian noise. Their work proved a similar $ 0-1 $ transition for the inner product of the top eigenvectors, which leads to a all-or-nothing phenomenon in the alignment problem, i.e. the accuracy of the recovery undergoes a sharp transition from $ 0 $ to $ 1 $ near some critical threshold.
However, the arguments in \cite{ganassali2022spectral} are not applicable in our case. Their proof heavily depends on the Gaussian assumption of the matrices, and the additive strucutre of the noise. In particular, they proof crucially relies on the orthogonal invariance of the Gaussian noise. While in our case, the noise is presented in terms of the resampling strength. There is no way to write the ``noise" in an additive form that is independent of the ``signal". Even in the Gaussian case, a rigorous analysis of the PCA-Recovery algorithm seems difficult.

Nevertheless, our results on the sensitivity of the eigenvector inner products suggest that, when $ k \gg n^{5/3} $, the two eigenvectors are approximately de-correlated so that they share almost no common information. Consequently, recovery of the data via aligning the principal components would be basically random guessing. Therefore, we conjecture that if $ k \gg n^{5/3} $, PCA-Recovery fails to recover the latent permutations in the sense that it can only achieve $ o(1) $ fraction of correct matching with the ground truth. On the other hand, when $ k \ll n^{5/3} $, the performance of our algorithm seems mysterious.


In Section \ref{sec:Numerical}, we empirically check the performance of the PCA-Recovery algorithm. Numerical simulations suggest that when $ k \gg n^{5/3} $, the performance of PCA-Recovery is indeed poor in the sense that the accuracy of the recovery is almost $ 0 $. On the other hand, when $ k \ll n^{5/3} $, experiments show that we cannot expect the sharp all-or-nothing phenomenon similarly as in \cite{ganassali2022spectral}.

Finally, we remark that what PCA-Recovery actually studies is a more difficult task, as we do not need direct observations of $ \bfX $ and $ \bfY $. We can consider a harder problem (in both statistical and computational sense), which we call alignment from covariance profile. In this problem, we only have access to the covariance between the samples and we aim to recover the correspondence between the samples from the two databases. A similar problem with Gaussian data and additive noise was considered in
\cite{wang2022random} as a prototype for matching random geometric dot-product graphs. The analysis of such database alignment problem with adversarial corruption will be an interesting direction for future studies.

\section{Numerical Experiments}\label{sec:Numerical}

We validate our theoretical prediction by checking the performance of PCA on synthetic data. To highlight the universality of our results, we will consider Gaussian data and Bernoulli data. The Gaussian data matrix consists of i.i.d. $ \calN(0,1) $ entries. The Bernoulli data matrix consists of i.i.d. entries taking value in $ \{\pm 1\} $ with equal probability $ \frac{1}{2} $. To visualize the stability-sensitivity transition, we focus on the overlap of the leading eigenvectors $ |\Iprod{\v}{\vk}| $ and $ |\Iprod{\u}{\uk}| $ as the observable. Note that, in the stability regime, the asymptotic colinearity \eqref{e.main2} implies that $ |\Iprod{\v}{\vk}|,|\Iprod{\u}{\uk}| \to 1 $. Therefore, we expect a phase transition from 1 to 0 at the critical threshold $ k \asymp n^{5/3} $.


We first focus on rectangular data matrices. For concreteness, we set $ p/n=0.25 $ and $ n=1000 $. As shown in Figure \ref{fig:Rectangle}, there is a clear phase transition for the eigenvector overlap varying from $ 1 $ to $ 0 $. Also, it provide good evidence that the transition happens at the critical threshold $ k \asymp n^{5/3} $.

\begin{figure}[h!]
     \centering
     \begin{subfigure}[b]{0.4\columnwidth}
         \centering
         \includegraphics[width=\textwidth]{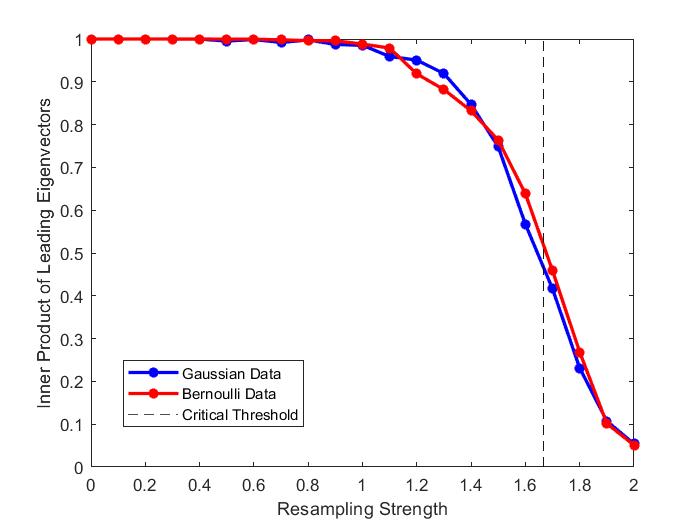}
         \caption{The inner product $ |\Iprod{\v}{\vk}| $}
         \label{fig:Rectangle_v}
     \end{subfigure}
     \begin{subfigure}[b]{0.4\columnwidth}
         \centering
         \includegraphics[width=\textwidth]{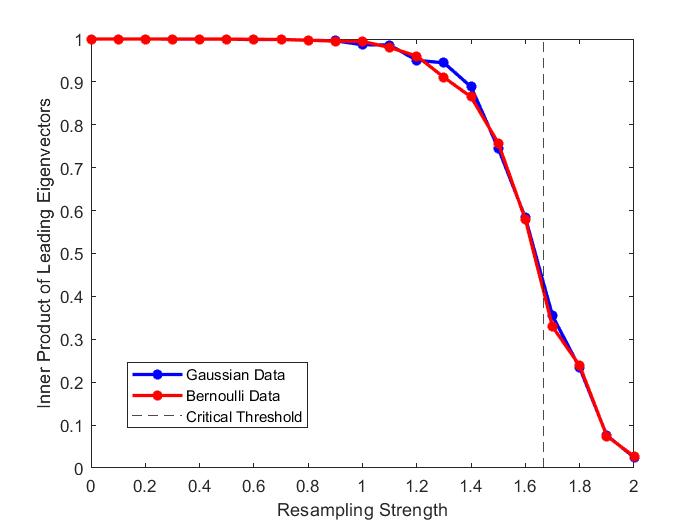}
         \caption{The inner product $ |\Iprod{\u}{\uk}| $}
         \label{fig:Rectangle_u}
     \end{subfigure}
        \caption{Inner products of the leading eigenvectors for $ 1000 \times 250 $ matrices with Gaussian and Bernoulli data. The horizontal axis is the resampling strength, given by $ \log_n(4k) $. Each experiment is averaged over 50 repetitions.}
        \label{fig:Rectangle}
\end{figure}

We also check square data $ 1000 \times 1000 $ matrices. As shown in Figure \ref{fig:Square}, for both Gaussian data and Bernoulli data, we have the same phase transition for the overlap of the leading eigenvectors. Again, this numerical simulation supports our theoretical prediction that the transition happens at the critical threshold $ k \asymp n^{5/3} $.

\begin{figure}[h!]
     \centering
     \begin{subfigure}[b]{0.4\columnwidth}
         \centering
         \includegraphics[width=\textwidth]{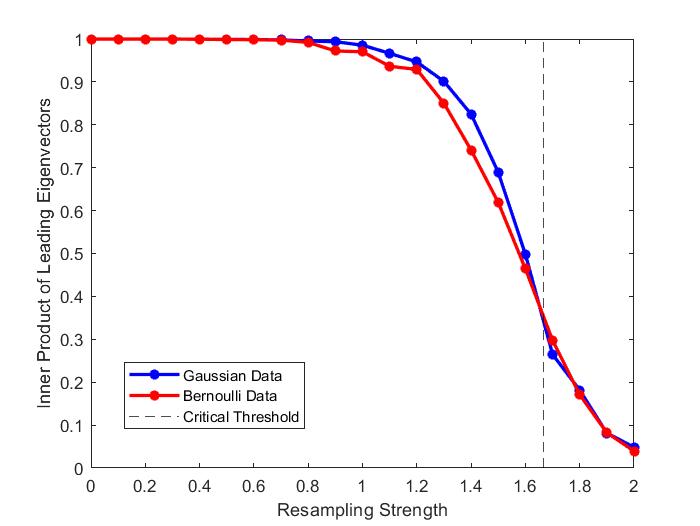}
         \caption{The inner product $ |\Iprod{\v}{\vk}| $}
         \label{fig:Sqaure_v}
     \end{subfigure}
     \begin{subfigure}[b]{0.4\columnwidth}
         \centering
         \includegraphics[width=\textwidth]{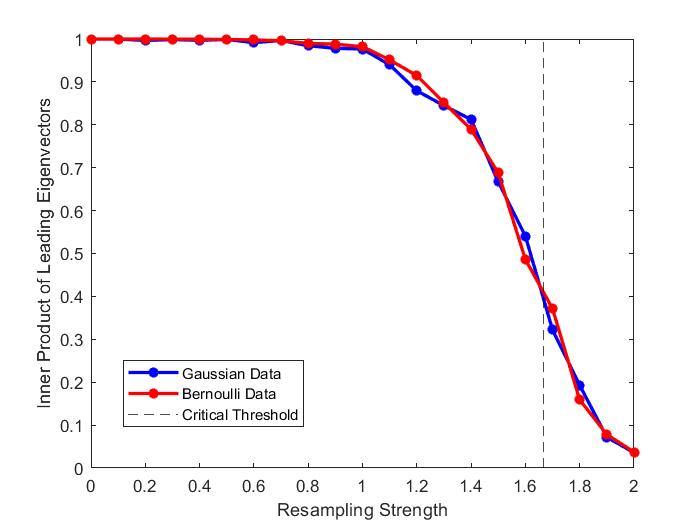}
         \caption{The inner product $ |\Iprod{\u}{\uk}| $}
         \label{fig:Square_u}
     \end{subfigure}
        \caption{Inner products of the leading eigenvectors for $ 1000 \times 1000 $ matrices with Gaussian and Bernoulli data. The horizontal axis is the resampling strength, given by $ \log_n(k ) $. Each experiment is averaged over 50 repetitions.}
        \label{fig:Square}
\end{figure}


Also we check the performance of the PCA-Recovery algorithm for database alignment. Still, we consider both Gaussian data and Bernoulli data. As shown in Figure \ref{fig:Recovery}, in both the rectangular case and the square case, the accuracy of the algorithm is almost $ 0 $ as long as the resampling strength $ k $ surpasses $ n^{5/3} $. This is consistent with the theoretical prediction that in this case the top eigenvectors approximately decorrelate. On the other hand, numerical simulations suggest that PCA-Recovery is brittle to resampling. In particular, we cannot expect a similar all-or-nothing phenomenon as in \cite{ganassali2022spectral}. Identifying the critical threshold below which PCA-Recovery can achieve almost perfect recovery is an interesting open problem for future research.

\begin{figure}[h!]
     \centering
     \begin{subfigure}[b]{0.4\columnwidth}
         \centering
         \includegraphics[width=\textwidth]{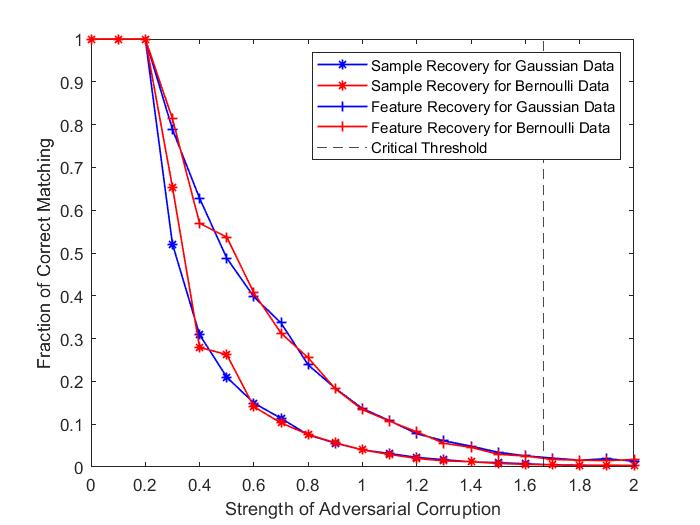}
         \caption{Fraction of correct matching for rectangular data}
         \label{fig:Recovery_Rectangle}
     \end{subfigure}
     \begin{subfigure}[b]{0.4\columnwidth}
         \centering
         \includegraphics[width=\textwidth]{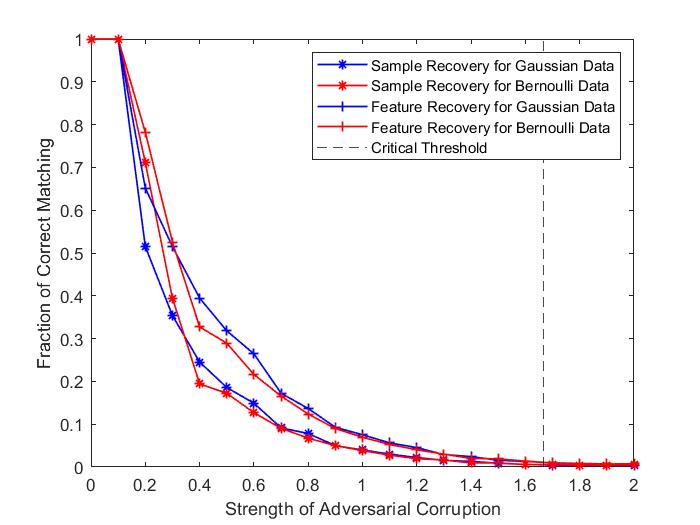}
         \caption{Fraction of correct matching for square data}
         \label{fig:Recovery_Square}
     \end{subfigure}
        \caption{Performance of the PCA algorithm for database alignment with adversarial corruption}
        \label{fig:Recovery}
\end{figure}

\appendix

\section{Notations and Organization}
We use $ C $ to denote generic constant, which may be different in each appearance. We denote $ A \lesssim B $ if there exists a universal $ C>0 $ such that $ A \leq CB $, and denote $ A \gtrsim B $ if $ A \geq CB $ for some universal $ C>0 $. We write $ A \asymp B $ if $ A \lesssim B $ and $ B \lesssim A $.

For the analysis of the sample covariance matrix, it is useful to apply the linearization trick (see e.g. \cite{tropp2012user,ding2018necessary,fan2018eigenvector}). Specifically, we will also analyze the symmetrization of $ \bfX $, which is defined as
\begin{equation}\label{eq:Symmetrization}
\tbfX := \left(
\begin{matrix}
0 & \bfX\\
\bfX^\top & 0
\end{matrix}
\right)
\end{equation}
The spectrum of the symmetrization $ \tbfX $ are given by the singular values $ \{\sqrt{\lambda_m}\}_{m=1}^p $, the symmetrized singular values $ \{-\sqrt{\lambda_m}\}_{m=1}^p $, and trivial eigenvalue $ 0 $ with multiplicity $ n-p $. Let $ \w_i:=(\u_i^\top,\v_i^\top)^\top \in \R^{n+p} $ be the concatenation of the eigenvectors $ \u_i $ and $ \v_i $. Then $ \w_i $ is the eigenvector of $ \tbfX $ associated with the eigenvalue $ \sqrt{\lambda_i} $. Indeed, we have
\begin{equation*}
\left(
\begin{matrix}
0 & \bfX\\
\bfX^\top & 0
\end{matrix}
\right) 
\left( 
\begin{matrix}
\u_i\\
\v_i
\end{matrix} 
\right) 
=
\left(
\begin{matrix}
\bfX v_i\\
\bfX^\top \u_i
\end{matrix}
\right)
=
\left(
\begin{matrix}
\sqrt{\lambda_i} \u_i\\
\sqrt{\lambda_i} \v_i
\end{matrix}
\right).
\end{equation*}

An important probabilistic concept that will be used repeatedly is the notion of overwhelming probability.

\begin{definition}[Overwhelming probability]
Let $ \{\calE_N\} $ be a sequence of events. We say that $ \calE_N $ holds with overwhelming probability if for any (large) $ D>0 $, there exists $ N_0(D) $ such that for all $ N \geq N_0(D) $ we have
$$ \P(\calE_N) \geq 1-N^{-D}. $$
\end{definition}


\paragraph{Organization}
The appendix is organized as follows. In Section \ref{sec:Prelim}, we collect some useful tools for the proof, including a variance formula for resampling and classical results from random matrix theory. In Section \ref{sec:Sensitivity}, we prove the sensitivity of PCA under excessive resampling. In Section \ref{sec:Stability}, we prove that PCA is stable if resampling of the data is moderate.

\section{Preliminaries}\label{sec:Prelim}

\subsection{Variance formula and resampling}
An essential technique for our proof is the formula for the variance of a function of independent random variables. This formula represents the variance via resampling of the random variables. This idea is first due to Chatterjee \cite{ChatterjeeThesis}, and in this paper we will use a slight extension of it as in \cite{bordenave2020noise}.

Let $ X_1,\cdots,X_N $ be independent random variables taking values in some set $ \mathcal{X} $, and let $ f: \mathcal{X}^N \to \R $ be some measurable function. Let $ X=\left( X_1,\cdots,X_N \right) $ and $ X' $ be an independent copy of $ X $. We denote
$$ X^{(i)}=(X_1,\cdots,X_{i-1},X_i',X_{i+1},\cdots,X_N),\ \ \mbox{and}\ \ X^{[i]}=(X_1',\cdots,X_i',X_{i+1},\cdots,X_N). $$
And in general, for $ A \subset [N] $, we define $ X^A $ to be the random vector obtained from $ X $ by replacing the components indexed by $ A $ by corresponding components of $ X' $. By a result of Chatterjee \cite{ChatterjeeThesis}, we have the following variance formula
$$ \Var\left( f(X) \right) = \dfrac{1}{2}\sum_{i=1}^N \EE \left[ \left( f(X) - f(X^{(i)}) \right)  \left( f(X^{[i-1]}) - f(X^{[i]}) \right) \right]. $$
We remark that this variance formula does not depend on the order of the random variables. Therefore, we can consider an arbitrary permutation of $ [N] $. Specifically, let $ \pi=(\pi(1),\cdots,\pi(N)) $ be a random permutation sampled uniformly from the symmetric group $ \calS_N $ and denote $ \pi([i]) := \{\pi(1),\cdots,\pi(i)\} $. Then we have
$$ \Var\left( f(X) \right) = \dfrac{1}{2}\sum_{i=1}^n \EE \left[ \left( f(X) - f(X^{(\pi(i))}) \right)  \left( f(X^{\pi([i-1])}) - f(X^{\pi([i])}) \right) \right]. $$
Note that, in the formula above, the expectation is taken with respect to both $ X $, $ X' $ and the random permutation $ \pi $.

Let $ j $ have uniform distribution over $ [N] $. Let $ X^{(j) \circ \pi([i-1])} $ denote the vector obtained from $ X^{\pi([i-1])} $ by replacing its $ j $-th component by another independent copy of the random variable $ X_j $ in the following way: If $ j $ belongs to $ \pi([i-1]) $, then we replace $ X_j' $ by $ X_j'' $; if $ j $ is not in $ \pi([i-1]) $, then we replace $ X_j $ by $ X_j''' $, where $ X'' $ and $ X''' $ are independent copies of $ X $ such that $ (X,X',X'',X''') $ are independent. With this notation, we have the following estimates.

\begin{lemma}[{Lemma 3 in \cite{bordenave2020noise}}]\label{l.VarianceBound}
Assume that $ j $ is chosen uniformly at random from the set $ [N]$ and independently of other random variables involved, we have for any $ k \in [N] $,
$$ B_k \leq \dfrac{2 \Var\left( f(X) \right)}{k}\left( \dfrac{N+1}{N}\right) $$
where for any $ i \in [N] $,
$$ B_i := \EE \left[ \left( f(X)-f(X^{(j)}) \right) \left( f(X^{\pi([i-1])}) - f(X^{(j)\circ\pi([i-1])}) \right) \right] $$
and the expectation is taken with respect to components of vectors, random permutations $ \pi $ and the random variable $ j $.
\end{lemma}



\subsection{Tools from random matrix theory}
In this section we collect some classical results in random matrix theory, which will be indispensable for proving the main theorems. These include concentration of the top eigenvalue, eigenvalue rigidity estimates, and eigenvector delocalization.

To begin with, we first state some basic settings and notations. It is well known that the empirical distribution of the spectrum of the sample covariance matrix converges to the Marchenko-Pastur distribution
\begin{equation}\label{eq:MP_Law}
\rho_{\MP}(x) = \dfrac{1}{2\pi \xi}\sqrt{\dfrac{[(x-\lambda_-)(\lambda_+ -x)]_+}{x^2}},
\end{equation}
where the endpoints of the spectrum are given by
\begin{equation}\label{e.EndPoints}
\lambda_{\pm} = (1 \pm \sqrt{\xi})^2.
\end{equation}
Define the typical locations of the eigenvalues:
$$ \gamma_m :=  \inf \left\{ E>0 : \int_{-\infty}^{E} \rho_{\MP}(x)\d x \geq \dfrac{m}{p} \right\},\ \ \ 1 \leq m \leq p. $$

A classical result in random matrix theory is the rigidity estimates of the eigenvalues \cite{pillai2014universality,bloemendal2014isotropic}. Let $ \hat{m}:=\min(m,p+1-k) $, for any small $ \eps>0 $ and large $ D>0 $ there exists $ n_0(\eps,D) $ such that the following holds for any $ n \geq n_0 $,
\begin{equation}\label{eq:Rigidity}
\P \left(|\lambda_m - \gamma_m| \leq  n^{-\frac{2}{3}+\eps}(\hat{m})^{-\frac{1}{3}} \  \mbox{for all}\ 1 \leq m \leq p \right) > 1-n^{-D}.
\end{equation}
We remark that the square case $ \xi \equiv 1 $ is actually significantly different, due to the singularity of the Marchenko-Pastur law at the spectral edge $ x=0 $. Near this edge, the typical eigenvalue spacing would be of order $ O(n^{-2}) $. In this case, it would be more convenient to state the rigidity estimates in terms of the singular values of the $ n \times n $ data matrix. The following result was proved in \cite{ajanki2014local}. For any small $ \eps>0 $ and large $ D>0 $, there exists $ n_0(\eps,D) $ such that the following is true for any $ n \geq n_0 $,
\begin{equation}\label{eq:Rigidity_Square}
\P \left( |\sqrt{\lambda_m} - \sqrt{\gamma_m}| \leq n^{-\frac{2}{3}+\eps} (n+1-m)^{-\frac{1}{3}} \  \mbox{for all}\ 1 \leq m \leq n \right) > 1-n^{-D}.    
\end{equation}
The intuition for this case is that the singular values of a square data matrix behaves like the eigenvalues of a Wigner matrix. More specifically, the singular values $ \{\sqrt{\lambda_m}\} $ and their symmetrization $ \{-\sqrt{\lambda_m}\} $ are the eigenvalues of the symmetrized matrix $ \tbfX $ defined in \eqref{eq:Symmetrization}, and $ \tbfX $ can be seen as a $ 2n \times 2n $ Wigner matrix with imprimitive variance profile (see \cite{ajanki2014local}). For more details, this was explained in \cite{wang2019quantitative,wang2022optimal}.

Another important result is the Tracy-Widom limit for the top eigenvalue (see e.g. \cite{pillai2014universality,ding2018necessary,wang2019quantitative,schnelli2021convergence}). Specifically,

\begin{lemma}\label{lem:Tracy-Widom}
For any $ \eps>0 $, with overwhelming probability, we have
$$ |\lambda - \lambda_+| \leq n^{-2/3+\eps},\ \ \mbox{and}\ \ \Var(\lambda) \lesssim  n^{-4/3}. $$
Moreover, for any $ \delta>0 $, there exists a constant $ c_0>0 $ such that
$$ \P \pth{\lambda_1 - \lambda_2 \geq c_0 n^{-2/3} } \geq 1-\delta. $$
\end{lemma}

\begin{proof}
The first result follows from the well-known Tracy-Widom limit for the top eigenvalue. Specifically,
$$ \lim_{n \to \infty}\P \pth{ \gamma p^{2/3}(\lambda - \lambda_+) \leq s} = F_1(s), $$
where $ \gamma $ is a constant depending only on the ratio $ \xi $ and $ F_1 $ is the type-1 Tracy-Widom distribution (in particular, \cite{wang2019quantitative,schnelli2021convergence} provided quantitative rate of convergence).
For the variance part, the Gaussian case was proved in \cite{ledoux2010small}, and the general case follows from universality, i.e. for any fixed $ m $
$$ \lim_{n \to \infty} \P \pth{ \pth{p^{2/3}(\lambda_\ell - \lambda_+) \leq s_\ell}_{1 \leq \ell \leq m} } = \lim_{n \to \infty} \P^{G} \pth{ \pth{p^{2/3}(\lambda_\ell - \lambda_+) \leq s_\ell}_{1 \leq \ell \leq m} }, $$
where $ \P^G $ denotes the probability measure associated with the Gaussian matrix.
The spectral gap estimate also follows from universality that the spectral statistics of the sample covariance matrix is the same as the Gaussian Orthogonal Ensemble (GOE), i.e. for any fixed $ m $
$$ \lim_{n \to \infty} \P \pth{ \pth{\gamma p^{2/3}(\lambda_\ell - \lambda_+) \leq s_\ell}_{1 \leq \ell \leq m} } = \lim_{n \to \infty} \P \pth{ \pth{p^{2/3}(\lambda_\ell^{GOE} - 2) \leq s_\ell}_{1 \leq \ell \leq m} } $$
For GOE, the desired spectral gap estimate was proved in e.g. \cite{anderson2010introduction}.
\end{proof}

Moreover, an estimate on the eigenvalue gap near the spectral edge is needed. The following result was proved in \cite{TaoVu,wang2012random}

\begin{lemma}\label{lem:Gap}
For any $ c>0 $, there exists $ \kappa>0 $ such that for every $ 1 \leq i \leq p $, with probability at least $ 1-n^{-\kappa} $, we have
$$ |\lambda_i - \lambda_{i+1}| \geq n^{-1-c}. $$
\end{lemma}

The property of eigenvectors is also a key ingredient for our proof. In particular, we extensively rely on the following delocalization property (see e.g. \cite{pillai2014universality,bloemendal2014isotropic,ding2018necessary,wang2022optimal})
\begin{lemma}\label{lem:Delocalization}
For any $ \eps>0 $, with overwhelming probability, we have
$$ \max_{1 \leq m \leq p} \|\v_m\|_\infty + \max_{1 \leq i \leq n} \|\u_i\|_\infty \leq n^{-\frac{1}{2}+\eps}. $$
\end{lemma}

\section{Proofs for the Sensitivity Regime}\label{sec:Sensitivity}

\subsection{Sensitivity analysis for neighboring data matrices}

As a first step, we will first show that resampling of a single entry has little perturbation effect on the top eigenvectors. This will be helpful to control the single entry resampling term in the variance formula (see Lemma \ref{l.VarianceBound}).

For any fixed $ 1 \leq i \leq n $ and $ 1 \leq \alpha \leq p $, let $ \bfX_{(i,\alpha)} $ be the matrix obtained from $ \bfX $ by replacing the $ (i,\alpha) $ entry $ \bfX_{i\alpha} $ with $ \bfX_{i\alpha}' $. Define the corresponding covariance matrix $ \bfH_{(i,\alpha)} := \bfX_{(i,\alpha)}^\top \bfX_{(i,\alpha)} $, and use $ \v^{(i,\alpha)} $ to denote its unit top eigenvector. Similarly, we denote by $ \u^{(i,\alpha)} $ the unit top eigenvector of $ \widehat{\bfH}_{(i,\alpha)} := \bfX_{(i,\alpha)} \bfX_{(i,\alpha)}^\top $.

\begin{lemma}\label{lem:Single}
Let $ c>0 $ small and $ 0<\delta < \frac{1}{2} - c $. For all $ 1 \leq i \leq n $ and $ 1 \leq \alpha \leq p $, on the event $ \{\lambda_1 - \lambda_2 \geq n^{-1-c}\} $, with overwhelming probability
\begin{equation}\label{eq:v_single}
\max_{i,\alpha} \min_{s \in \sth{\pm 1}} \|\v - s \v^{(i,\alpha)}\|_\infty \leq n^{-\frac{1}{2}-\delta}
\end{equation}
and similarly
$$ \max_{i,\alpha} \min_{s \in \sth{\pm 1}} \| \u - s \u^{(i,\alpha)} \|_\infty \leq n^{-\frac{1}{2}-\delta} $$
\end{lemma}

\begin{proof}
Let $ \lambda_1^{(i,\alpha)} \geq \cdots \geq \lambda_p^{(1,\alpha)} $ denote the eigenvalues of the matrix $ \bfH_{(i,\alpha)} $, and let $ \v_\beta^{(i,\alpha)} $ denote the unit eigenvector associated with the eigenvalue $ \lambda_\beta^{(i,\alpha)} $. Similarly, we define the unit eigenvectors $ \{\u_j^{(i,\alpha)}\} $ for the matrix $ \widehat{\bfH}_{(i,\alpha)} $. Using the variational characterization of the eigenvalues, we have
$$ \lambda_1 \geq \Iprod{\v_1^{(i,\alpha)}}{\bfH \v_1^{(i,\alpha)}} = \lambda_1^{(i,\alpha)} + \Iprod{\v_1^{(i,\alpha)}}{ (\bfH - \bfH_{(i,\alpha)}) \v_1^{(i,\alpha)}}. $$
Since $ \bfX $ and $ \bfX_{(i,\alpha)} $ differ only at the $ (i,\alpha) $ entry, we have
\begin{equation*}
(\bfH - \bfH_{(i,\alpha)})_{\beta \gamma} = (\bfX^\top \bfX - \bfX_{(i,\alpha)}^\top \bfX_{(i,\alpha)})_{\beta \gamma} = \begin{cases}
(\bfX_{i\alpha} - \bfX_{i\alpha}')\bfX_{i\gamma}\ \ &\mbox{if}\ \beta=\alpha, \gamma \neq \alpha,\\
(\bfX_{i\alpha} - \bfX_{i\alpha}')\bfX_{i\beta}\ \ &\mbox{if}\ \beta \neq \alpha, \gamma=\alpha,\\
\bfX_{i\alpha}^2 - (\bfX_{i\alpha}')^2\ \ &\mbox{if}\ \beta=\alpha,\gamma=\alpha,\\
0\ \ &\mbox{otherwise}.
\end{cases}
\end{equation*}
Thus, setting $ \Delta_{i\alpha} := \bfX_{i\alpha} - \bfX_{i\alpha}' $, we have
\begin{equation}\label{eq:H_Diff_Bilinear}
\begin{aligned}
& \Iprod{\v_1^{(i,\alpha)}}{ (\bfH - \bfH_{(i,\alpha)}) \v_1^{(i,\alpha)}}\\
& \quad = 2 \v_1^{(i,\alpha)}(\alpha) \Delta_{i\alpha} \pth{\sum_{\beta=1}^p (\bfX_{(i,\alpha)})_{i\beta} \v_1^{(i,\alpha)}(\beta) - \bfX_{i\alpha}' \v_1^{(i,\alpha)}(\alpha) } + \pth{\v_1^{(i,\alpha)}(\alpha)}^2 (\bfX_{i\alpha}^2 - (\bfX_{i\alpha}')^2)\\
& \quad = 2 \v_1^{(i,\alpha)}(\alpha) \Delta_{i\alpha} \pth{\bfX_{(i,\alpha)} \v_1^{(i,\alpha)}}(i) + \pth{\v_1^{(i,\alpha)}(\alpha)}^2 \pth{\bfX_{i\alpha}^2 - (\bfX_{i\alpha}')^2 - 2(\bfX_{i\alpha} - \bfX_{i\alpha}') \bfX_{i\alpha}' }\\
& \quad = 2 \sqrt{\lambda_1^{(i,\alpha)}} \v_1^{(i,\alpha)}(\alpha) \u_1^{(i,\alpha)}(i) \Delta_{i\alpha} + \pth{\v_1^{(i,\alpha)}(\alpha)^2} \Delta_{i\alpha}^2.
\end{aligned}
\end{equation}
This gives us
\begin{equation}\label{eq:lambda-lambda_ialpha}
\lambda_1 \geq \lambda_1^{(i,\alpha)} - 2 \sqrt{\lambda_1^{(i,\alpha)}} \pth{|\bfX_{i\alpha}| + |\bfX_{i\alpha}'|} \|\v_1^{(i,\alpha)}\|_\infty \|\u_1^{(i,\alpha)}\|_\infty - \pth{|\bfX_{i\alpha}| + |\bfX_{i\alpha}'|}^2 \|\v_1^{(i,\alpha)}\|_\infty^2.
\end{equation}
Similarly,
\begin{equation}\label{eq:lambda_ialpha-lambda}
\lambda_1^{(i,\alpha)} \geq \lambda_1 - 2 \sqrt{\lambda_1} \pth{|\bfX_{i\alpha}| + |\bfX_{i\alpha}'|} \|\v_1\|_\infty \|\u_1\|_\infty - \pth{|\bfX_{i\alpha}| + |\bfX_{i\alpha}'|}^2 \|\v_1\|_\infty^2.
\end{equation}
From the sub-exponential decay assumption \eqref{e.Assumption2}, we know $ |\bfX_{i\alpha}|,|\bfX_{i\alpha}'| \leq n^{-1/2+\eps} $ with overwhelming probability for any $ \eps>0 $. Moreover, by the delocalization of eigenvectors, with overwhelming probability, we have
$$ \max \pth{ \|\v_1\|_\infty, \|\u_1\|_\infty, \|\v_1^{(i,\alpha)}\|_\infty, \|\u_1^{(i,\alpha)}\|_\infty } \leq n^{-\frac{1}{2}+\eps} $$
Moreover, by the rigidity estimates \eqref{eq:Rigidity}, with overwhelming probability we have
$$ |\lambda_1 - \lambda_+| \leq n^{-\frac{2}{3}+\eps},\ \ |\lambda_1^{(i,\alpha)} - \lambda_+| \leq n^{-\frac{2}{3} + \eps} $$
Therefore, combining with a union bound, the above two inequalities \eqref{eq:lambda-lambda_ialpha} and \eqref{eq:lambda_ialpha-lambda} together yield
\begin{equation}\label{eq:lambda_single_diff}
\max_{1 \leq i \leq n, 1 \leq \alpha \leq p} |\lambda_1 - \lambda_1^{(i,\alpha)}| \leq n^{-3/2 + 3\eps}
\end{equation}
with overwhelming probability.

We write $ \v_1^{(i,\alpha)} $ in the orthonormal basis of eigenvectors $ \sth{\v_\beta} $:
$$ \v_1^{(i,\alpha)} = \sum_{\beta=1}^p a_\beta \v_\beta.  $$
Using this representation and the spectral theorem,
$$  \sum_{\beta=1}^p \lambda_\beta a_\beta \v_\beta = \bfH \v_1^{(i,\alpha)} = \pth{\bfH - \bfH_{(i,\alpha)}} \v_1^{(i,\alpha)} + \pth{\lambda_1^{(i,\alpha)} - \lambda_1} \v_1^{(i,\alpha)} + \lambda_1 \v_1^{(i,\alpha)}. $$
As a consequence,
$$ \lambda_1 \v_1^{(i,\alpha)} = \sum_{\beta=1}^p \lambda_\beta \a_\beta \v_\beta + \pth{\bfH_{(i,\alpha)} - \bfH} \v_1^{(i,\alpha)} + \pth{\lambda_1 - \lambda_1^{(i,\alpha)}} \v_1^{(i,\alpha)}. $$
For $ \beta \neq 1 $, taking inner product with $ \v_\beta $ yields
$$ \lambda_1 a_\beta = \Iprod{\v_\beta}{\lambda_1 \v_1^{(i,\alpha)}} = \lambda_\beta a_\beta + \Iprod{\v_\beta}{(\bfH_{(i,\alpha)} - \bfH) \v_1^{(i,\alpha)}} + (\lambda_1 - \lambda_1^{(i,\alpha)}) a_\beta, $$
which implies
\begin{equation}\label{eq:Coefficient}
\pth{(\lambda_1 - \lambda_\beta) + (\lambda_1^{(i,\alpha)} - \lambda_1)} a_\beta = \Iprod{\v_\beta}{(\bfH_{(i,\alpha)} - \bfH) \v_1^{(i,\alpha)}}.
\end{equation}
By a similar computation as in \eqref{eq:H_Diff_Bilinear}, we have
\begin{equation}\label{eq:H_Diff_Bilinear_Bound}
\begin{aligned}
\abs{ \Iprod{\v_\beta}{(\bfH_{(i,\alpha)} - \bfH) \v_1^{(i,\alpha)}} } &= \abs{ \Delta_{i\alpha} \pth{ \sqrt{\lambda_1^{(i,\alpha)}} \v_\beta(\alpha) \u_1^{(i,\alpha)}(i) + \sqrt{\lambda_\beta} \v_1^{(i,\alpha)}(\alpha) \u_\beta(i) } }\\
&\lesssim \pth{|\bfX_{i\alpha}| + |\bfX_{i\alpha}'|} \pth{ \|\v_\beta\|_\infty \|\u_1^{(i,\alpha)}\|_\infty + \|\v_1^{(i,\alpha)}\|_\infty \|\u_\beta\|_\infty }\\
&\leq n^{-\frac{3}{2} + 3 \eps}
\end{aligned}
\end{equation}
with overwhelming probability, where the second step follows from rigidity of eigenvalues and the last step follows from the sub-exponential decay assumption and delocalization of eigenvectors.

Consider the event $ \calE := \{\lambda_1 - \lambda_2 \geq n^{-1-c}\} $. Fix some $ \omega>0 $ small. By rigidity of eigenvalues \eqref{eq:Rigidity}, on the event $ \calE $, with overwhelming probability
\begin{equation}\label{eq:Eigenvalue_Diff}
\lambda_1 - \lambda_\beta \gtrsim 
\begin{cases}
n^{-1-c} & \mbox{if}\ \ 2 \leq \beta \leq n^\omega,\\
\beta^{2/3} n^{-2/3} & \mbox{if}\ \ n^\omega < \beta \leq p.
\end{cases}
\end{equation}
On the event $ \calE $, using \eqref{eq:lambda_single_diff}, \eqref{eq:H_Diff_Bilinear_Bound} and \eqref{eq:Eigenvalue_Diff}, with overwhelming probability we have
\begin{equation}\label{eq:Coefficient_Bound}
|a_\beta| \lesssim
\begin{cases}
n^{-\frac{1}{2} + c + 3\eps} & \mbox{if}\ \ 2 \leq \beta \leq n^\omega,\\
\beta^{-\frac{2}{3}} n^{-\frac{5}{6} + 3\eps} & \mbox{if}\ \ n^\omega < \beta \leq p.
\end{cases}
\end{equation}
Choose $ s = a_1/|a_1| $. Note that $ 1-|a_1| \leq \sum_{\beta=2}^p |a_\beta| $.
Thanks to the delocalization of eigenvectors, with overwhelming probability, we have
\begin{equation*}
\|s \v_1 - \v_1^{(i,\alpha)}\|_\infty =  \bigg\| (s-a_1)\v_1 + \sum_{\beta=2}^p a_\beta \v_\beta \bigg\|_\infty \leq (1-|a_1|)\|\v_1\|_\infty + \sum_{\beta=2}^p |a_\beta| \|\v_\beta\|_\infty \leq n^{-\frac{1}{2}+\eps} \sum_{\beta=2}^p |a_\beta|.
\end{equation*}
Thus, on the event $ \calE $, it follows from \eqref{eq:Eigenvalue_Diff} that
\begin{align*}
\|s \v_1 - \v_1^{(i,\alpha)}\|_\infty &\lesssim n^{-\frac{1}{2} + \eps} \bigg( n^{-\frac{1}{2}+3\eps + c +\omega} + n^{-\frac{5}{6} + 3\eps} \sum_{\beta=n^\omega}^p \beta^{-\frac{2}{3}} \bigg)\\
&\lesssim n^{-1+4\eps+c+\omega} + n^{-1+4\eps}.
\end{align*}
Choosing $ \eps $ and $ \omega $ small enough so that $ 4\eps+c+\omega < \frac{1}{2} $, we conclude that \eqref{eq:v_single} is true.

A similar bound for $ \u $ can be shown by the same arguments for $ \widehat{\bfH}=\bfX \bfX^\top $. Hence, we have shown the desired results.
\end{proof}

\subsection{Proof of Theorem \ref{thm:main1}}\label{sec:Proof_Sensitivity}

Now we are ready to prove the resampling sensitivity.

Let $ \bfX'' \in \R^{n \times p} $ be a copy of $ \bfX $ that is independent of $ \bfX $ and $ \bfX' $. For an arbitrary index $ (i,\alpha) $ with $ 1 \leq i \leq n $ and $ 1 \leq \alpha \leq p $, we introduce another random matrix $ \bfY_{(i,\alpha)} $ obtained from $ \bfX $ by replacing the $ (i,\alpha) $ entry $ \bfX_{i\alpha} $ by $ \bfX_{i\alpha}'' $. Similarly, we denote $ \Yk_{(i,\alpha)} $ the matrix obtained via the same procedure from $ \Xk $. For the matrix $ \Xk $, we do the similar resampling procedure in the following way: if $ (i,\alpha) \in S_k $, then replace $ \Xk_{i\alpha} $ with $ \bfX_{i\alpha}'' $; if $ (i,\alpha) \notin S_k $, then replace $ \Xk_{i\alpha} $ with $ \bfX_{i\alpha}''' $, where $ \bfX''' $ is another independent copy of $ \bfX $, $ \bfX' $ and $ \bfX'' $.

In the following analysis, we choose an index $ (s,\theta) $ uniformly at random from the set of all pairs $ \sth{(i,\alpha): 1 \leq i \leq n,1 \leq \alpha \leq p} $. Let $ \mu $ be the top singular value of $ \bfY_{(s,\theta)} $ and use $ \f \in \R^n $ and $ \g \in \R^p $ to denote the normalized top left and right singular vectors of $ \bfY_{(s,\theta)} $. Similarly, we define $ \Muk $, $ \fk $ and $ \gk $ for $ \Yk_{(s,\theta)} $. We also denote by $ \h $ and $ \hk $ the concatenation of $ \f,\g $ and $ \fk,\gk $, respectively. Finally, let $ \tXk $, $ \tbfY $ and $ \tYk $ be the symmetrization \eqref{eq:Symmetrization} of $ \Xk $, $ \bfY $ and $ \Yk $, respectively. When the context is clear, we will omit the index $ (s,\theta) $ for the convenience of notations.

\paragraph{Step 1.}
By Lemma \ref{l.VarianceBound}, we have
\begin{equation}\label{eq:Variance_Singular}
\frac{2 \Var(\sigma)}{k} \cdot \frac{np+1}{np} \geq \EE \qth{ (\sigma-\mu) \pth{ \Sigk -\Muk } }.
\end{equation}
Using the variational characterization of the top singular value
\begin{equation*}
\Iprod{\f}{\bfX \g} \leq \sigma = \Iprod{\u}{\bfX \v},\ \ \  \Iprod{\u}{\bfY \v} \leq \mu = \Iprod{\f}{\bfY \g}.
\end{equation*}
This implies
\begin{equation}\label{eq:Test_v_g}
\Iprod{\f}{(\bfX - \bfY) \g} \leq \sigma-\mu \leq \Iprod{\u}{(\bfX - \bfY) \v}.
\end{equation}
Applying the same arguments to $ \Xk $ and $ \Yk $, we have
\begin{equation*}
\iprod{\fk}{\pth{\Xk - \Yk} \gk} \leq \Sigk-\Muk \leq \iprod{\uk}{\pth{\Xk - \Yk} \vk}.
\end{equation*}
Since the matrices $ \bfX $ and $ \bfY $ differ only at the $ (s,\theta) $ entry, for any vectors $ \a \in \R^n $ and $ \b \in \R^p $ we have
$$ \Iprod{\a}{(\bfX-\bfY)\b} = \Delta_{s\theta}\,\a(s)\b(\theta),\ \ \iprod{\a}{\pth{\Xk-\Yk}\b} = \Dk_{s\theta}\, \a(s)\b(\theta), $$
where 
\begin{equation*}
\Delta_{s\theta}:=\bfX_{s\theta} - \bfX_{s\theta}'',\ \ \mbox{and}\ \   \Dk_{s\theta} :=
\begin{cases}
\bfX_{s\theta}' -\bfX_{s\theta}'' & \mbox{if}\ (s,\theta) \in S_k,\\
\bfX_{s\theta} - \bfX_{s\theta}''' & \mbox{if}\ (s,\theta) \notin S_k.
\end{cases}
\end{equation*}
Therefore,
$$ \Delta_{s\theta}\, \f(s)\g(\theta) \leq \sigma-\mu \leq \Delta_{s\theta} \u(s)\v(\theta),\ \ \Dk_{s\theta}\, \fk(s)\gk(\theta) \leq \Sigk-\Muk \leq \Dk_{s\theta} \uk(s)\vk(\theta). $$
Consider
$$ T_1 := \pth{ \Delta_{s\theta} \u(s)\v(\theta)} \pth{\Dk_{s\theta} \uk(s)\vk(\theta)},\ \ T_2 := \pth{\Delta_{s\theta}\, \f(s)\g(\theta)} \pth{\Dk_{s\theta}\, \fk(s)\gk(\theta)}, $$
$$ T_3 := \pth{\Delta_{s\theta} \u(s)\v(\theta)} \pth{\Dk_{s\theta}\, \fk(s)\gk(\theta)},\ \ T_4:= \pth{\Delta_{s\theta}\, \f(s)\g(\theta)} \pth{\Dk_{s\theta} \uk(s)\vk(\theta)}. $$
Then we have
\begin{equation}\label{eq:Singular_Diff_Prod}
\min(T_1,T_2,T_3,T_4) \leq (\sigma-\mu)\pth{\Sigk-\Muk} \leq \max(T_1,T_2,T_3,T_4).
\end{equation}

To estimate \eqref{eq:Singular_Diff_Prod}, we introduce the following events
\begin{equation}\label{eq:Event_Delocalization}
\calE_1 := \sth{ \max \pth{ \|\v\|_\infty, \|\u\|_\infty, \|\f\|_\infty, \|\g\|_\infty, \|\vk\|_\infty, \|\uk\|_\infty, \|\fk\|_\infty, \|\gk\|_\infty } \leq n^{-\frac{1}{2}+\eps} },   
\end{equation}
and
\begin{equation}\label{eq:Event_Single_Perturbation}
\calE_2 := \sth{ \max \pth{ \|\v-\g\|_\infty, \|\u-\f\|_\infty, \|\vk-\gk\|_\infty, \|\uk-\fk\|_\infty } \leq n^{-\frac{1}{2}-\delta} }.   
\end{equation}
Define the event $ \calE := \calE_1 \cap \calE_2 $. 
On the event $ \calE $, for all
$$ J \in \sth{ \u(s)\v(\theta)\uk(s)\vk(\theta), \u(s)\v(\theta)\fk(s)\gk(\theta), \f(s)\g(\theta)\uk(s)\vk(\theta), \f(s)\g(\theta)\fk(s)\gk(\theta) } $$
we have
\begin{equation}\label{eq:J_Approx}
\abs{J- \u(s)\v(\theta)\uk(s)\vk(\theta)} = O \pth{ n^{-2-\delta+3\eps} }.
\end{equation}
Let $ T:=\min(T_1,T_2,T_3,T_4) $. On the event $ \calE $, using \eqref{eq:J_Approx} we have
\begin{equation}\label{eq:T_on_E}
T \geq \pth{\Delta_{s\theta} \Dk_{s\theta}} \u(s)\v(\theta)\uk(s)\vk(\theta) - O \pth{  \abs{\Delta_{s\theta} \Dk_{s\theta}} n^{-2-\delta+3\eps} }.
\end{equation}

\paragraph{Step 2.}
Next we claim that the contribution of $ T $ when $ \calE $ does not hold is negligible. Specifically, we have
\begin{equation}\label{eq:T_Complement}
\EE \qth{ T \1_{\calE^c} } = o(n^{-3}).    
\end{equation}
Without loss of generality, it suffices to show that
\begin{equation}\label{eq:T_Complement_Equiv}
\EE \qth{ \Delta_{s\theta} \, \Dk_{s\theta} \u(s)\v(\theta)\uk(s)\vk(\theta) \1_{\calE^c} } = o(n^{-3}).   
\end{equation}
To see this, using $ \1_{\calE^c} = \1_{\calE_1 \backslash \calE} + \1_{\calE_1^c} $, we decompose the expectation into two parts
$$ \EE \qth{ \Delta_{s\theta} \, \Dk_{s\theta} \u(s)\v(\theta)\uk(s)\vk(\theta) \1_{\calE^c} } = I_1 + I_2, $$
where
\begin{equation*}
I_1 := \EE \qth{ \Delta_{s\theta} \, \Dk_{s\theta} \u(s)\v(\theta)\uk(s)\vk(\theta) \1_{\calE_1 \backslash \calE} }, \ \ 
I_2 := \EE \qth{ \Delta_{s\theta} \, \Dk_{s\theta} \u(s)\v(\theta)\uk(s)\vk(\theta) \1_{\calE_1^c} }.    
\end{equation*}
For the first term $ I_1 $, by delocalization and the relation $ \calE_1 \backslash \calE = \calE_1 \backslash \calE_2 $, we have
\begin{equation}\label{eq:I1_E2_Complement}
|I_1| \leq n^{-2+4\eps} \EE \qth{ \abs{\Delta_{s\theta} \, \Dk_{s\theta}} \1_{\calE_1 \backslash \calE_2} } \leq n^{-2+4\eps} \EE \qth{ \abs{\Delta_{s\theta} \, \Dk_{s\theta}} \1_{\calE_2^c} }.
\end{equation}
Note that the random variable $ \Delta_{s\theta} \, \Dk_{s\theta} $ and the event $ \calE_2 $ are dependent. To decouple these variables, we introduce a new event. 
Consider the event $ \calE_3 := \calA \cup \calB $, where
$$ \calA := \sth{ \min \pth{ \sigma_1-\sigma_2 , \Sigk_1 - \Sigk_2 } \geq n^{-1-c} },\ \ \calB := \sth{ \min \pth{ \mu_1-\mu_2 , \Muk_1 - \Muk_2 } \geq n^{-1-c} } $$
Then,
\begin{align*}
\EE \qth{ \abs{\Delta_{s\theta} \, \Dk_{s\theta}} \1_{\calE_3^c} } &\lesssim \EE \qth{ \pth{ \Delta_{s\theta}^2 + (\Dk_{s\theta})^2 } \1_{\calE_3^c} }\\
&\lesssim \EE \qth{ \pth{ \bfX_{s\theta}^2 + (\bfX_{s\theta}')^2 + (\bfX_{s\theta}'')^2 + (\bfX_{s\theta}''')^2 } \1_{\calE_3^c} }\\
&\lesssim \EE \qth{ \pth{ \bfX_{s\theta}^2 + (\bfX_{s\theta}')^2 } \1_{\calB^c} } + \EE \qth{ \pth{ (\bfX_{s\theta}'')^2 + (\bfX_{s\theta}''')^2 } \1_{\calA^c} }.
\end{align*}
Observe that the random variables $ \bfX_{s\theta} $ and $ \bfX_{s\theta}' $ are independent of the event $ \calB $, and the random variable $ \bfX_{s\theta}'' $ is independent of $ \calA $. 
Therefore, by Lemma \ref{lem:Gap},
$$ \EE \qth{ \pth{ \bfX_{s\theta}^2 + (\bfX_{s\theta}')^2 } \1_{\calB^c} } = O(n^{-1-\kappa}),\ \ \   \EE \qth{ \pth{ (\bfX_{s\theta}'')^2 + (\bfX_{s\theta}''')^2 } \1_{\calA^c} } = O(n^{-1-\kappa}). $$
By Lemma \ref{lem:Single}, we have $ \P(\calE_3 \backslash \calE_2) = O(N^{-D}) $ for any fixed large $ D>0 $. Using the Cauchy-Schwarz inequality, we have
\begin{align*}
\EE \qth{ \abs{\Delta_{s\theta} \, \Dk_{s\theta}} \1_{\calE_2^c} } &= \EE \qth{ \abs{\Delta_{s\theta} \, \Dk_{s\theta}} \1_{\calE_3^c} } + \EE \qth{ \abs{\Delta_{s\theta} \, \Dk_{s\theta}} \1_{\calE_3 \backslash \calE_2} }\\
&= O(n^{-1-\kappa}) +  \sqrt{ \EE \qth{\abs{ \Delta_{s\theta} \, \Dk_{s\theta} }^2} } \sqrt{\P(\calE_3 \backslash \calE_2)}\\
&= O(n^{-1-\kappa}) + O(N^{-D})\\
&= O(n^{-1-\kappa}).
\end{align*}
Choosing $ 4\eps <\kappa $, then \eqref{eq:I1_E2_Complement} yields
$$ |I_1| \leq O(n^{-2+4\eps-1-\kappa}) = o(n^{-3}). $$
For the term $ I_2 $, note that $ \u $, $ \v $, $ \uk $ and $ \vk $ are unit vectors. We have that 
$$ \max(\|\u\|_\infty, \|\v\|_\infty, \|\uk\|_\infty, \|\vk\|_\infty) \leq 1. $$
Recall that $ \calE_1 $ holds with overwhelming probability. By the Cauchy-Schwarz inequality, for any large $ D>0 $, we have
$$ |I_2| \leq \EE \qth{ \abs{\Delta_{s\theta} \, \Dk_{s\theta}} \1_{\calE_1^c} } \leq \sqrt{ \EE \qth{\abs{\Delta_{s\theta} \, \Dk_{s\theta}}^2} } \sqrt{\P(\calE_1^c)} = O(N^{-D}). $$
Hence we have shown the desired claim \eqref{eq:T_Complement_Equiv}.

\paragraph{Step 3.}
Combining \eqref{eq:Singular_Diff_Prod}, \eqref{eq:T_on_E} and \eqref{eq:T_Complement}, we obtain
$$ \EE \qth{ (\sigma-\mu) \pth{ \Sigk -\Muk } } \geq \EE \qth{ \Delta_{s\theta} \Dk_{s\theta} \, \u(s)\v(\theta)\uk(s)\vk(\theta) } + o(n^{-3}). $$
Since $ \frac{np+1}{np} \leq 2 $, by \eqref{eq:Variance_Singular} we have
\begin{equation}\label{eq:Inner_Prod_Variance}
\EE \qth{ \Delta_{s\theta} \Dk_{s\theta} \, \u(s)\v(\theta)\uk(s)\vk(\theta) } \leq \frac{4 \Var(\sigma)}{k} + o(n^{-3}).    
\end{equation}
Since the random index $ (s,\theta) $ is uniformly sampled, we have
\begin{equation}\label{eq:Expectation_Index}
\EE \qth{ \Delta_{s\theta} \Dk_{s\theta} \, \u(s)\v(\theta)\uk(s)\vk(\theta) } = \frac{1}{np} \EE \qth{ \sum_{1\leq i \leq n,1\leq \alpha \leq p} \Delta_{i\alpha} \Dk_{i\alpha} \, \u(i)\v(\alpha)\uk(i)\vk(\alpha) }.
\end{equation}
Note that
\begin{equation*}
\Delta_{i\alpha} \Dk_{i\alpha} = 
\begin{cases}
(\bfX_{i\alpha} - \bfX_{i\alpha}')(\bfX_{i\alpha}' - \bfX_{i\alpha}'') & \mbox{if}\ (i,\alpha) \in S_k,\\
(\bfX_{i\alpha} - \bfX_{i\alpha}')(\bfX_{i\alpha} - \bfX_{i\alpha}''') & \mbox{if}\ (i,\alpha) \notin S_k.
\end{cases}
\end{equation*}
In either case, we have $ \EE [\Delta_{i\alpha} \Dk_{i\alpha} ] = p^{-1} $. Therefore,
\begin{equation*}
\sum_{1 \leq i \leq n,1 \leq \alpha \leq p} \EE \qth{\Delta_{i\alpha} \Dk_{i\alpha} \left| \right. S_k} \u(i)\v(\alpha)\uk(i)\vk(\alpha) = \frac{1}{p} \Iprod{\v}{\vk} \Iprod{\u}{\uk}.    
\end{equation*}
Consequently, this implies
\begin{equation}\label{eq:Inner_Prod_Formula}
\EE \qth{ \sum_{1 \leq i \leq n,1 \leq \alpha \leq p} \EE \qth{\Delta_{i\alpha} \Dk_{i\alpha} \left.\right| S_k} \u(i)\v(\alpha)\uk(i)\vk(\alpha) } = \frac{1}{p} \EE \qth{ \Iprod{\v}{\vk} \Iprod{\u}{\uk} }.   
\end{equation}
Moreover, we claim that
\begin{equation}\label{eq:Inner_Prod_Conditional_Error}
\EE \qth{ \sum_{1 \leq i \leq n,1 \leq \alpha \leq p} \pth{ \Delta_{i\alpha} \Dk_{i\alpha} - \EE \qth{\Delta_{i\alpha} \Dk_{i\alpha} \left.\right| S_k} } \u(i)\v(\alpha)\uk(i)\vk(\alpha) } =o(n^{-1}).   
\end{equation}
For the ease of notations, we set $ \Xi_{i\alpha} := \Delta_{i\alpha} \Dk_{i\alpha} - \EE [\Delta_{i\alpha} \Dk_{i\alpha} | S_k]  $. It suffices to show that for all pairs $ (i,\alpha) $ we have
\begin{equation}\label{eq:Inner_Prod_Conditional_Equiv1}
\EE \qth{ \Xi_{i\alpha} \u(i)\v(\alpha)\uk(i)\vk(\alpha) } = o(n^{-3}).
\end{equation}
To see this, we introduce another copy of $ \bfX $, denoted by $ \bfX'''' $, which is independent of all previous random variables $ (\bfX,\bfX',\bfX'',\bfX''') $. For an arbitrarily fixed index $ (i,\alpha) $, we define matrices $ \hatX_{(i,\alpha)} $ and $ \hatXk_{(i,\alpha)} $ by resampling the $ (i,\alpha) $ entry of $ \bfX $ and $ \Xk $ with $ \bfX_{i\alpha}'''' $. Let $ \hatu $, $ \hatv $ be the left and right top singular vector of $ \hatX $, and similarly $ \hatuk $, $ \hatvk $ for $ \hatXk $. Define
$$ \psi_{i\alpha} := \u(i)\v(\alpha)\uk(i)\vk(\alpha),\ \ \hatpsi_{i\alpha} := \hatu(i)\hatv(\alpha)\hatuk(i)\hatvk(\alpha). $$
A crucial observation is that $ \Xi_{i\alpha} $ and $ \hatpsi_{i\alpha} $ are independent. This is because, conditioned on $ S_k $, the matrices $ \hatX $ and $ \hatXk $ are independent of $ (\bfX_{i\alpha},\bfX_{i\alpha}',\bfX_{i\alpha}'',\bfX_{i\alpha}''') $. Such a conditional independence is also true for the singular vectors, and hence also holds for $ \hatpsi_{i\alpha} $. On the other hand, by definition, the variable $ \Xi_{i\alpha} $ only depends on $ (\bfX_{i\alpha},\bfX_{i\alpha}',\bfX_{i\alpha}'',\bfX_{i\alpha}''') $. Therefore,
$$ \EE \qth{\Xi_{i\alpha} \hatpsi_{i\alpha}} = \EE \qth{ \EE [ \Xi_{i\alpha} | S_k ] \,  \EE [ \hatpsi_{i\alpha} | S_k ] }  =0 $$
Thus, we reduce \eqref{eq:Inner_Prod_Conditional_Equiv1} to showing
\begin{equation}\label{eq:Inner_Prod_Conditional__Equiv2}
\EE \qth{ \Xi_{i\alpha} \pth{\psi_{i\alpha} - \hatpsi_{i\alpha}} } = o(n^{-3}).   
\end{equation}
The proof of \eqref{eq:Inner_Prod_Conditional__Equiv2} is similar as previous arguments. Consider the events
$$ \hat{\calE}_1 := \sth{ \max \pth{ \|\v\|_\infty, \|\u\|_\infty, \|\hatu\|_\infty, \|\hatv\|_\infty, \|\vk\|_\infty, \|\uk\|_\infty, \|\hatuk\|_\infty, \|\hatvk\|_\infty } \leq n^{-\frac{1}{2}+\eps} }, $$
$$ \hat{\calE}_2 := \sth{ \max \pth{ \|\v-\hatv\|_\infty, \|\u-\hatu\|_\infty, \|\vk-\hatvk\|_\infty, \|\uk-\hatuk\|_\infty } \leq n^{-\frac{1}{2}-\delta} }. $$
On the event $ \hat{\calE} := \hat{\calE}_1 \cap \hat{\calE}_2 $, we have $ |\psi_{i\alpha} - \hatpsi_{i\alpha}| = O(n^{-2-\delta+3\eps}) $. Note that $ \EE [|\Xi_{i\alpha}|] = O(n^{-1}) $ since $ \EE [|\Delta_{i\alpha} \Dk_{i\alpha}|] = O(n^{-1}) $. As a consequence,
\begin{equation}\label{eq:Xi_On}
\EE \qth{ \abs{\Xi_{i\alpha} (\psi_{i\alpha} - \hatpsi_{i\alpha}) } \1_{\hat{\calE}} } = O(n^{-3-\delta+3\eps}) = o(n^{-3}).
\end{equation}
Using the same argument as in \eqref{eq:T_Complement_Equiv}, we have
\begin{equation}\label{eq:Xi_Complement}
\EE \qth{ \abs{\Xi_{i\alpha} (\psi_{i\alpha} - \hatpsi_{i\alpha}) } \1_{\hat{\calE}^c} } \lesssim N^{-2+4\eps} \EE \qth{|\Xi_{i\alpha}| \1_{\hat{\calE}^c}} = O(N^{-2+4\eps} N^{-1-\kappa}) = o(n^{-3}),
\end{equation}
where $ \kappa $ is the constant in the gap property (Lemma \ref{lem:Gap})
Thus, by \eqref{eq:Xi_On} and \eqref{eq:Xi_Complement}, we have shown the desired claim \eqref{eq:Inner_Prod_Conditional__Equiv2}.

Based on \eqref{eq:Inner_Prod_Variance} and \eqref{eq:Expectation_Index}, combining \eqref{eq:Inner_Prod_Formula} and \eqref{eq:Inner_Prod_Conditional_Error} yields
$$ \frac{1}{np^2} \EE \qth{  \Iprod{\v}{\vk} \Iprod{\u}{\uk}  } \leq \frac{4 \Var(\sigma)}{k} + o\pth{\frac{1}{n^3}} + o\pth{\frac{1}{n^2 p}} $$
By Lemma \ref{lem:Tracy-Widom} and the assumption $ k \gg n^{5/3} $, we have
$$ \EE \qth{  \Iprod{\v}{\vk} \Iprod{\u}{\uk}  } \leq \frac{np^2}{k} O(n^{-4/3}) + o(1) = o(1). $$
This implies
\begin{equation}\label{eq:Cross_Inner_Prod}
\EE \qth{ | \Iprod{\v}{\vk} \Iprod{\u}{\uk} | } \to 0.
\end{equation}



\paragraph{Step 4.}
Consider the symmetrization matrix $ \tbfX $ defined in \eqref{eq:Symmetrization}. The variational representation of the top eigenvalue yields
$$ \sigma = \frac{\Iprod{\w}{\tbfX \w}}{\|\w\|_2^2}, \ \ \Sigk = \frac{\Iprod{\wk}{\tXk \wk}}{\|\wk\|_2^2}\ \ \mbox{with}\ \|\w\|_2^2 = \|\wk\|_2^2 =2. $$
Using the same arguments as in Step 1-3, we can conclude that
$$ \EE \qth{ | \Iprod{\w}{\wk} |^2 } = \EE \qth{ |\Iprod{\v}{\vk} + \Iprod{\u}{\uk}|^2 }  \to 0 $$
Combined with \eqref{eq:Cross_Inner_Prod}, this gives us
$$ \EE \qth{ |\Iprod{\v}{\vk}|^2 + |\Iprod{\u}{\uk}|^2 } \to 0, $$
which proves the desired results.

\section{Proofs for the Stability Regime}\label{sec:Stability}




Throughout the whole section, we will focus on the behaviour of $ \v $ and $ \vk $. Similar results also hold for $ \u $ and $ \uk $ via the same arguments.

\subsection{Linearization and local law of resolvent}

As mentioned in the introduction, in certain cases it would be more convenient to consider the symmetrization $ \tbfX $ of the matrix $ \bfX $ (defined as in \eqref{eq:Symmetrization}) when studying its spectral properties. For $ z \in \C $ with $ \im z>0 $,
We introduce the resolvent of this symmetrization
\begin{equation}\label{eq:Resolvent}
\bfR(z) := \left(
\begin{matrix}
-\bfI_{n} & \bfX\\
\bfX^\top & -z \bfI_{p}
\end{matrix}
\right)^{-1}.    
\end{equation}
Note that $ \bfR(z) $ is not the conventional definition of the resolvent matrix, but we still call it resolvent for convenience.
For the ease of notations, we will relabel the indices in $ \bfR $ in the following way:
\begin{definition}[Index sets]
We define the index sets
$$ \calI_1 := \sth{1,\dots,n},\ \ \calI_2 := \sth{1,\dots,p},\ \ \calI := \calI_1 \cup \sth{n+\alpha:\alpha \in \calI_2}. $$
For a general matrix $ \mathbf{M} \in \R^{|\calI| \times |\calI|} $, we label the indices of the matrix elements in the following way: for $ a,b \in \calI $, if $ 1 \leq a,b \leq n $, then typically we use Latin letters $ i,j $ to represent them; if $ n+1 \leq a,b \leq n+p $, we use the corresponding Greek letters $ \alpha = a-n $ and $ \beta = b-n $ to represent them.
\end{definition}

The resolvent $ \bfR $ is closely related to the eigenvalue and eigenvectors of the covariance matrix. As discussed in \cite{ding2018necessary}[Equation (3.9),(3.10)], we have
\begin{equation}\label{eq:Resolvent_Spectral}
\bfR_{ij}(z) = \sum_{\ell=1}^n \frac{z \u_\ell(i) \u_\ell(j)}{\lambda_\ell - z},\ \ \bfR_{\alpha\beta}(z) = \sum_{\ell=1}^p \frac{\v_\ell(\alpha) \v_\ell(\beta)}{\lambda_\ell - z},    
\end{equation}
and
$$ \bfR_{i\alpha}(z) = \sum_{\ell=1}^p \frac{\sqrt{\lambda_\ell} \u_\ell(i) \v_\ell(\alpha) }{\lambda_\ell - z},\ \ \bfR_{\alpha i}(z) = \sum_{\ell=1}^p \frac{\sqrt{\lambda_\ell} \v_\ell(\alpha) \u_\ell(i)}{\lambda_\ell -z}. $$
An important result is the local Marchenko-Pastur law for the resolvent matrix $ \bfR $. This was first proved in \cite{ding2018necessary}[Lemma 3.11], and we also refer to \cite{hwang2019local}[Proposition 2.13] for a version that can be directly applied. Specifically, the resolvent matrix $ \bfR $ has a deterministic limit, defined by
\begin{equation}
\bfG(z) :=
\left(
\begin{matrix}
-(1+m_\MP(z))^{-1} \bfI_{n} & 0\\
0 & m_\MP(z) \bfI_p
\end{matrix}
\right),
\end{equation}
where $ m_\MP(z) $ is the Stieltjes transform of the Marchenko-Pasture law \eqref{eq:MP_Law}, given by
$$ m_{\MP}(z) := \int_{\R} \dfrac{\rho_{\MP}(x)}{x-z}\d x=\dfrac{1-\xi-z+\sqrt{(z-\lambda_-)(z-\lambda_+)}}{2\xi z}, $$
where $ \sqrt{\cdot} $ denotes the square root on the complex plane whose branch cut is the negative real line. With this choice we always have $ \im m_{\MP}(z)>0 $ when $ \im z>0 $.

To state the local law, we will focus on the spectral domain
\begin{equation}\label{eq:Spectral_Domain}
\calS := 
\begin{cases}
\sth{E+\ii \eta : \frac{\lambda_-}{2} \leq E \leq \lambda_++1, 0 <\eta<3 } & \mbox{if}\ 0<\xi<1.\\
\sth{E+\ii \eta : \frac{1}{10} \leq E \leq \lambda_++1, 0<\eta<3 } & \mbox{if}\ \xi=1.
\end{cases}
\end{equation}

\begin{lemma}[Local Marchenko-Pastur law]\label{lem:Local_law}
For any $ \eps>0 $, the following estimate holds With overwhelming probability uniformly for $ z \in \calS $,
\begin{equation}\label{eq:Local_law}
\max_{a,b \in \calI} \abs{\bfR_{ab}(z) - \bfG_{ab}(z)} \leq n^\eps \pth{\sqrt{\frac{\im m_\MP(z)}{n \eta}} + \frac{1}{n\eta}}.    
\end{equation}
\end{lemma}

To give a precise characterization of the resolvent, we rely on the following estimates for the Stieltjes transform $ m_\MP(z) $ of the Marchenko-Pasture law. We refer to e.g. \cite{bloemendal2016principal}[Lemma 3.6] for more details.

\begin{lemma}[Estimate for $ m_\MP(z) $]
For $ z =E + \ii \eta $, let $ \kappa(z) := \min(|E-\lambda_-|,|E-\lambda_+|) $ denote the distance to the spectral edge. If $ z \in \calS $, we have
\begin{equation}
|m_\MP(z)| \asymp 1,\ \ \mbox{and}\ \ \ 
\im m_\MP(z) \asymp 
\begin{cases}
\sqrt{\kappa(z) + \eta} & \mbox{if}\ E \in [\lambda_-,\lambda_+],\\
\frac{\eta}{\sqrt{\kappa(z) + \eta}} & \mbox{if}\ E \notin [\lambda_-,\lambda_+]. 
\end{cases}
\end{equation}
\end{lemma}

In the following analysis, we will work with $ z=E+\ii \eta $ satisfying $ |E-\lambda_+| \leq n^{-2/3+\delta} $ and $ \eta=n^{-2/3 - \delta} $, where $ 0<\delta<\frac{1}{3} $ is some parameter. Uniformly in this regime, the local law yields that the following is true with overwhelming probability for all $ \eps>0 $ and some universal constant $ C_0>0 $,
\begin{equation}\label{eq:Local_Law_Edge}
\sup_{z} \max_{a \neq b \in \calI} |\bfR_{ab}(z)| \leq n^{-\frac{1}{3} + \delta + \eps}  ,\ \ \mbox{and}\ \ \sup_{z} \max_{a \in \calI} |\bfR_{aa}(z)| \leq  C_0.
\end{equation}
These estimates will be used repeatedly in the following subsections.

\subsection{Stability of the resolvent}

In this subsection, we will prove the main technical result for the proof of resampling stability. Specifically, we will show that under moderate resampling, the resolvent matrices are stable. Since resolvent is closely related to various spectral statistics, this stability lemma for resolvent will be a key ingredient for our proof.

\begin{lemma}\label{lem:Stability_Resolvent}
Assume $ k \leq n^{5/3-\epsilon_0} $ for some $ \epsilon_0>0 $. There exists $ \delta_0>0 $ such that for all $ 0<\delta<\delta_0 $, uniformly for $ z=E+\ii\eta $ with $ |E-\lambda_+| \leq n^{-2/3+\delta} $ and $ \eta=n^{-2/3-\delta} $, there exists $ c>0 $ such that the following is true with overwhelming probability
\begin{equation}\label{eq:Resolvent_Stability}
\max_{\alpha,\beta} \abs{ \Rk_{\alpha\beta}(z) - \bfR_{\alpha\beta}(z) } \leq \frac{1}{n^{1+c} \eta},\ \ \max_{i,j} \abs{ \Rk_{ij}(z) - \bfR_{ij}(z) } \leq \frac{1}{n^{1+c} \eta}.
\end{equation}
\end{lemma}

\begin{proof}
Recall that $ S_k := \{(i_1,\alpha_1),\dots,(i_k,\alpha_k)\} $ is the random subset of matrix indices whose elements are resampled in the matrix $ \bfX $. For $ 1 \leq t \leq k $, let $ \Xt $ be the matrix obtained from $ \bfX $ by resampling the $ \{(i_s,\alpha_s)\}_{1 \leq s \leq t} $ entries and let $ \calF_t $ be the $ \sigma $-algebra generated by the random variables $ \bfX $, $ S_k $ and $ \{\bfX_{i_s \alpha_s}'\}_{1 \leq s \leq t} $. For $ a,b \in \calI $, we define
$$ T_{ab} := \sth{t:\{i_t,\alpha_t\} \cap \{a,b\} \neq \emptyset  }. $$

Let $ \eps>0 $ be an arbitrarily fixed parameter, and let $ C_0 $ be the constant in \eqref{eq:Local_Law_Edge}. Consider the event $ \calE_t \in \calF_t $ where for all $ z=E+\ii\eta $ with $ |z-\lambda_+| \leq n^{-2/3+\delta} $ and $ \eta=n^{-2/3-\delta} $ we have
$$ \max_{a \neq b} \abs{\Rt_{ab}(z)} \leq  n^{-\frac{1}{3} + \delta + \eps} =: \Psi,\ \ \mbox{and}\ \ \max_{a} \abs{\Rt_{aa}(z)} \leq C_0.  $$

Define $ \Xt_0 $ as the matrix obtained from $ \Xt $ by replacing the $ (i_t,\alpha_t) $ entry with $ 0 $, and also define its symmetrization $ \tXt_0 \in \R^{|\calI| \times |\calI|} $ as in \eqref{eq:Symmetrization}. Note that $ \tXtt_0 $ is $ \calF_t $-measurable. We write
$$ \tXt = \tXtt_0 + \Ptt ,\ \ \  \tXtt = \tXtt_0 + \Qtt, $$
where $ \Pt, \Qt  $ are $ |\calI| \times |\calI| $ symmetric matrices whose elements are all $ 0 $ except at the $ (i_{t},\alpha_{t}) $ and $ (\alpha_{t},i_{t}) $ entries, satisfying
$$ 
(\Pt)_{ab} = \begin{cases}
\bfX_{i_{t}\alpha_{t}} & \mbox{if}\  \sth{a,b}=\{i_{t},\alpha_{t}\},\\
0 & \mbox{otherwise}
\end{cases}
\ \ \ 
(\Qt)_{ab} = \begin{cases}
\bfX_{i_{t}\alpha_{t}}' & \mbox{if}\  \sth{a,b}=\{i_{t},\alpha_{t}\},\\
0 & \mbox{otherwise}
\end{cases}
.
$$

Define the resolvents for the matrices $ \tXt $ and $ \tXt_0 $ as in \eqref{eq:Resolvent}:
\begin{equation*}
\Rt := 
\left(
\begin{matrix}
-\bfI_n & \Xt\\
(\Xt)^\top & -z\bfI_p
\end{matrix}
\right)^{-1}
,\ \ 
\Rt_0 :=
\left(
\begin{matrix}
-\bfI_n & \Xt_0\\
(\Xt_0)^\top & -z\bfI_p
\end{matrix}
\right)^{-1}
.
\end{equation*}
Using first-order resolvent expansion, we obtain
\begin{equation}\label{eq:Res_Exp_1}
\Rtt_0 = \Rt + \Rt \Ptt \Rt + \pth{\Rt \Ptt}^2 \Rtt_0.
\end{equation}
The triangle inequality yields
$$ \abs{\pth{\Rtt_0 - \Rt}_{\alpha\beta}} \leq \abs{ \pth{\Rt \Ptt \Rt}_{\alpha\beta} } + \abs{ \pth{\Rt \Ptt \Rt \Ptt \Rtt_0}_{\alpha\beta} }. $$
Note that $ \Ptt $ has only two non-zero entries,
$$ \pth{\Rt \Ptt \Rt}_{\alpha\beta} = \sum_{\ell_1,\ell_2} \Rt_{\alpha \ell_1} \Ptt_{\ell_1 \ell_2} \Rt_{\ell_2 \beta} = X_{i_{t+1}\alpha_{t+1}} \pth{ \Rt_{\alpha i_{t+1}} \Rt_{\alpha_{t+1} \beta} + \Rt_{\alpha \alpha_{t+1}} \Rt_{i_{t+1} \beta} }  $$
Recall that $ |X_{i_{t+1}\alpha_{t+1}}| \leq n^{-1/2+\eps} $ with overwhelming probability thanks to the sub-exponential decay assumption \eqref{e.Assumption2}. Then on the event $ \calE_t $, we have
$$ \abs{ \pth{\Rt \Ptt \Rt}_{\alpha\beta} } \leq 2C_0 \Psi n^{-\frac{1}{2}+\eps}. $$
Similarly,
$$ \pth{\Rt \Ptt \Rt \Ptt \Rtt_0}_{\alpha\beta} = \sum_{\{m_1,m_2\},\{m_3,m_4\} = \{i_{t+1},\alpha_{t+1}\} } \Rt_{\alpha m_1} \Ptt_{m_1 m_2} \Rt_{m_2 m_3} \Ptt_{m_3 m_4} (\Rtt_0)_{m_4 \beta}. $$
We use the trivial bound $ |\Rtt_0| \leq \eta^{-1} $ for the last term. Then, on the event $ \calE_t $, we have
$$ \abs{ \pth{\Rt \Ptt \Rt \Ptt \Rtt_0}_{\alpha\beta} } \leq 2 n^{-1+2\eps} \eta^{-1} (\Psi^2 + C_0^2) \ll \Psi. $$
Therefore, we have shown that, on the event $ \calE_t $,
\begin{equation}\label{eq:Rtt0_Apriori}
\max_{\alpha \neq \beta} \abs{ (\Rtt_0)_{\alpha\beta} } \leq 2 \Psi,\ \ \max_{\alpha} \abs{ (\Rtt_0)_{\alpha\alpha} } \leq 2 C_0.
\end{equation}
Similarly, using the first-order resolvent expansion for $ \Rtt $ around $ \Rt $, we have
$$ \Rtt = \Rt + \Rt (\Ptt-\Qtt) \Rt + \pth{\Rt (\Ptt-\Qtt)}^2 \Rtt. $$
By the same arguments as above, on the event $ \calE_t $, we can derive
$$ \max_{\alpha \neq \beta} \abs{ \Rtt_{\alpha\beta} } \leq 2 \Psi,\ \ \max_{\alpha} \abs{ \Rtt_{\alpha\alpha} } \leq 2 C_0. $$

Next, we use the resolvent identity (or zeroth-order expansion) for $ \Rtt $ and $ \Rtt_0 $:
$$ \Rtt = \Rtt_0 - \Rtt_0 \Qtt \Rtt.  $$
This leads to
$$ \abs{ \pth{\Rtt - \Rtt_0}_{\alpha\beta} } = \abs{ \sum_{\{\ell_1,\ell_2\}=\{i_{t+1}\alpha_{t+1}\}} (\Rtt_0)_{\alpha \ell_1} \Qtt_{\ell_1 \ell_2} \Rtt_{\ell_2 \beta} } $$
Thus, on the event $ \calE_t $, we conclude
\begin{equation}\label{eq:frakf}
\abs{ \pth{\Rtt - \Rtt_0}_{\alpha\beta} } \leq 4 n^{-\frac{1}{2}+\eps} \pth{ \Psi^2 + C_0\Psi \1_{((t+1) \in T_{\alpha\beta})} } =: \frakf_{\alpha\beta}^{[t+1]}
\end{equation}
Meanwhile, the second-order resolvent expansion of $ \Rtt $ around $ \Rtt_0 $ yields
$$ \Rtt = \Rtt_0 - \Rtt_0 \Qtt \Rtt_0 + \pth{\Rtt_0 \Qtt}^2 \Rtt_0 - \pth{\Rtt_0 \Qtt}^3 \Rtt. $$
A key observation is that $ \Rtt_0 $ is $ \calF_t $-measurable, and $ \EE [\Qtt | \calF_t] =0 $.
For simplicity of notations, we set
$$ \frakq_{\alpha\beta}^{[t]} := \pth{ (\Rt_0 \tbfE^{(i_t,\alpha_t)})^2 \Rt_0 }_{\alpha \beta} $$
where $ \tbfE^{(i_t,\alpha_t)} \in \R^{|\calI| \times |\calI|} $ is the symmetrization of the matrix $ \bfE^{(i_t,\alpha_t)} \in \R^{n \times p} $ whose elements are all $ 0 $ except $ \bfE^{(i_t,\alpha_t)}_{i_t \alpha_t}=1 $.
Then we have
\begin{equation}\label{eq:Rtt_Rtt0}
\abs{ \EE\qth{\Rtt_{\alpha\beta} | \calF_t} - (\Rtt_0)_{\alpha\beta} - p^{-1} \frakq_{\alpha\beta}^{[t+1]} } \leq 32 n^{-\frac{3}{2}+3\eps} \pth{\Psi^2 C_0^2 + C_0^4 \1_{((t+1) \in T_{\alpha\beta})}} =: \frakg_{\alpha \beta}^{[t+1]}.
\end{equation}

Similarly, using resolvent expansion of $ \Rt $ around $ \Rtt_0 $, we obtain
$$ \Rt = \Rtt_0 - \Rtt_0 \Ptt \Rtt_0 + (\Rtt_0 \Ptt)^2 \Rtt_0 - (\Rtt_0 \Ptt)^3 \Rt. $$
By the same arguments as above, on the event $ \calE_t $, we deduce that
\begin{equation}\label{eq:Rt_Rtt0}
\abs{ \Rt_{\alpha\beta} - (\Rtt_0)_{\alpha\beta} + \bfX_{i_{t+1}\alpha_{t+1}} \frakp_{\alpha\beta}^{[t+1]} - \bfX_{i_{t+1}\alpha_{t+1}}^2 \frakq_{\alpha\beta}^{[t+1]} } \leq \frakg_{\alpha\beta}^{[t+1]}
\end{equation}
where
\begin{equation}\label{eq:frakp}
\frakp_{\alpha\beta}^{[t]} := \pth{\Rt_0 \tbfE^{(i_t,\alpha_t)}\Rt_0}_{\alpha\beta}.
\end{equation}
Combining \eqref{eq:Rtt_Rtt0} and \eqref{eq:Rt_Rtt0} yields
\begin{equation}\label{eq:Remainder}
\abs{ \EE \qth{\Rtt_{\alpha\beta} | \calF_t} - \Rt_{\alpha\beta} - \bfX_{i_{t+1}\alpha_{t+1}}\frakp_{\alpha\beta}^{[t+1]} + (\bfX_{i_{t+1}\alpha_{t+1}}^2 - p^{-1}) \frakq_{\alpha\beta}^{[t+1]} } \leq 2 \frakg_{\alpha\beta}^{[t+1]}.
\end{equation}
By a telescopic summation, we obtain
\begin{equation}\label{eq:Telescopic}
\begin{aligned}
\Rk_{\alpha\beta} - \bfR_{\alpha\beta} &= \sum_{t=0}^{k-1} \pth{ \Rtt_{\alpha\beta} - \Rt_{\alpha\beta} }\\
&= \sum_{t=0}^{k-1} \pth{\Rtt_{\alpha\beta} - \EE \qth{\Rtt_{\alpha\beta} | \calF_t}} + \sum_{t=0}^{k-1} \bfX_{i_{t+1}\alpha_{t+1}} \frakp_{\alpha\beta}^{[t+1]}  + \sum_{t=0}^{k-1} (\bfX_{i_{t+1}\alpha_{t+1}}^2 - p^{-1}) \frakq_{\alpha\beta}^{[t+1]}   + \mathfrak{r}_{\alpha\beta} 
\end{aligned}
\end{equation}
where the remainder $ \mathfrak{r}_{\alpha\beta} $ is bounded by \eqref{eq:Remainder}
$$ |\mathfrak{r}_{\alpha\beta}| \leq 2\sum_{t=0}^{k-1} \frakg_{\alpha\beta}^{[t+1]}. $$
Recall the expression of $ \frakg_{\alpha\beta}^{[t]} $, to estimate the remainder, we need to control the size of the set $ T_{\alpha\beta} $.
Note that $ \EE \qth{|T_{\alpha\beta}|} = 2k/p. $
By a Berstein-type inequality (see e.g. \cite{ChatterjeePTRF}[Proposition 1.1]), for any $ x>0 $, we have
$$ \P \pth{|T_{\alpha\beta}| \geq \EE \qth{|T_{\alpha\beta}|} + x} \leq \exp \pth{- \frac{x^2}{4 \EE \qth{|T_{\alpha\beta}|} + 2x}} $$
Recall that $ k \leq n^{5/3-\epsilon_0} $. The inequality together with a union bound implies that 
$$ \max_{\alpha,\beta} |T_{\alpha\beta}| \leq \frac{3 \max(k,p(\log n)^2)}{p} =: \sfT $$ 
with overwhelming probability. We denote this event by $ \calT $. On the event $ \calT $, we have
\begin{equation}\label{eq:Remainder_Bound}
|\mathfrak{r}_{\alpha\beta}| \leq 2k n^{-\frac{3}{2}+3\eps} \Psi^2 C_0^2 + 2 n^{-\frac{3}{2}+3\eps} C_0^4 \sfT \leq n^{3\eps} \sqrt{\sfT} \Psi^2.
\end{equation}

For the first term in \eqref{eq:Telescopic}, we set
$$ \frakw_{\alpha\beta}^{[t+1]} := \pth{\Rtt_{\alpha\beta} - \EE \qth{\Rtt_{\alpha\beta} | \calF_t}} \1_{\calE_t}. $$
Note that $ \calE_t \in \calF_t $. This implies that $ \EE[\frakw_{\alpha\beta}^{[t+1]}|\calF_t]=0 $. Moreover, by \eqref{eq:frakf}, on the event $ \calE_t $ we have $ |\frakw_{\alpha\beta}^{[t+1]}| \leq 2 \frakf_{\alpha\beta}^{[t+1]} $. Further, on the event $ \calT $,
$$ \pth{\sum_{t=0}^{k-1} (\frakf_{\alpha\beta}^{[t+1]})^{2}}^{1/2} \leq n^{-\frac{1}{2}+\eps} \Psi^2 \sqrt{k} + n^{-\frac{1}{2}+\eps} C_0 \Psi \sqrt{\sfT} \leq 2 n^{\eps} \Psi^2 \sqrt{\sfT}. $$
Using the Azuma-Hoeffding inequality, for any $ x \geq 0 $, we have
$$ \P \pth{ \abs{ \sum_{t=0}^{k-1} \frakw_{\alpha\beta}^{[t+1]} } \geq 2 n^{\eps} \Psi^2 \sqrt{\sfT} x } \leq 2 \exp \pth{-\frac{x^2}{2}}. $$
Moreover,
$$ \P \pth{ \abs{\sum_{t=0}^{k-1} \pth{\Rtt_{\alpha\beta} - \EE \qth{\Rtt_{\alpha\beta} | \calF_t} } } \geq 2 n^{\eps} \Psi^2 \sqrt{\sfT} x } \leq \P \pth{ \abs{ \sum_{t=0}^{k-1} \frakw_{\alpha\beta}^{[t+1]} } \geq 2 n^{\eps} \Psi^2 \sqrt{\sfT} x } + \sum_{t=0}^{k-1}\P(\calE_t^c). $$
Recall that $ \calE_t $ holds with overwhelming probability, and consequently $ \sum_{t=0}^{k-1} \P (\calE_t^c) \leq n^{-D} $ for any $ D>0 $. Choosing $ x=n^\eps $ implies that with overwhelming probability
\begin{equation}\label{eq:frakw_bound}
\abs{\sum_{t=0}^{k-1} \pth{\Rtt_{\alpha\beta} - \EE \qth{\Rtt_{\alpha\beta} | \calF_t} } }  \leq 2 n^{2\eps} \Psi^2 \sqrt{\sfT}. 
\end{equation}

For the next two terms in \eqref{eq:Telescopic}, we will deal with them by introducing a backward filtration. Let $ \calF_t' $ be the $ \sigma $-algebra generated by the random variables $ \bfX' $, $ S_k $ and $ \{\bfX_{i\alpha}\} $ with $ i \notin \sth{i_1,\dots,i_t} $ and $ \alpha \notin \sth{\alpha_1,\dots,\alpha_t} $. Similarly as above, we consider the event $ \calE_t' $ that for all $ z=E+\ii\eta $ with $ |z-\lambda_+| \leq n^{-2/3+\delta} $ and $ \eta=n^{-2/3-\delta} $ we have
$$ \max_{a \neq b} \abs{\Rt_{ab}(z)} \leq \Psi,\ \ \mbox{and}\ \ \max_{a} \abs{\Rt_{aa}(z)} \leq C_0. $$
Using resolvent expansion, the same arguments for \eqref{eq:Rtt0_Apriori} yield that, on the event $ \calE_t' $, we have
$$ \max_{\alpha \neq \beta} \abs{(\Rt_0)_{\alpha\beta}} \leq 2 \Psi,\ \ \max_{\alpha} \abs{(\Rt_0)_{\alpha\alpha}} \leq 2C_0. $$
A key observation is that $ \frakp_{\alpha\beta}^{[t]} $ defined in \eqref{eq:frakp} is $ \calF_t' $-measurable. Also, we have $ \EE [\bfX_{i_t \alpha_t} | \calF_t']=0 $. Consider 
$$ \widetilde{\frakp}_{\alpha\beta}^{[t]} := \bfX_{i_t \alpha_t} \frakp_{\alpha\beta}^{[t]} \1_{\calE_t'}. $$
Then we have $ \EE [\widetilde{\frakp}_{\alpha\beta}^{[t]} | \calF_t']=0 $ since we also have $ \calE_t' \in \calF_t' $. Note that
$$ \P \pth{ \abs{ \sum_{t=0}^{k-1} \bfX_{i_{t+1}\alpha_{t+1}} \frakp_{\alpha\beta}^{[t+1]} } \geq x } \leq \P \pth{ \abs{ \sum_{t=0}^{k-1} \widetilde{\frakp}_{\alpha\beta}^{[t+1]} } \geq x } + \sum_{t=0}^{k-1} \P((\calE_{t+1}')^c),  $$
The second term is negligible since $ \calE_t' $ holds with overwhelming probability. To estimate the first term, we use Azuma-Hoeffding inequality as before.
Based on similar arguments as in \eqref{eq:frakf}, we deduce
$$ \abs{\widetilde{\frakp}_{\alpha\beta}^{[t]}} \leq 4 n^{-\frac{1}{2}+\eps} \pth{ \Psi^2 + C_0\Psi \1_{(t \in T_{\alpha\beta})} }. $$
By considering the event $ \calT $ and using Azuma-Hoeffding inequality as in \eqref{eq:frakw_bound}, we can conclude that with overwhelming probability,
$$ \abs{ \sum_{t=0}^{k-1} \widetilde{\frakp}_{\alpha\beta}^{[t+1]} } \leq n^{2\eps} \Psi^2 \sqrt{\sfT} $$
As a consequence, with overwhelming probability
\begin{equation}\label{eq:frakp_bound}
\abs{ \sum_{t=0}^{k-1} \bfX_{i_{t+1}\alpha_{t+1}} \frakp_{\alpha\beta}^{[t+1]} } \lesssim n^{2\eps} \Psi^2 \sqrt{\sfT}. 
\end{equation}
For the third term in \eqref{eq:Telescopic}, by the same arguments, we have
\begin{equation}\label{eq:frakq_bound}
\abs{ \sum_{t=0}^{k-1} (\bfX_{i_{t+1}\alpha_{t+1}}^2 - p^{-1}) \frakq_{\alpha\beta}^{[t+1]} } \lesssim n^{2\eps} \Psi^2 \sqrt{\sfT}. \end{equation}

Finally, combining \eqref{eq:Telescopic}, \eqref{eq:Remainder_Bound}, \eqref{eq:frakw_bound}, \eqref{eq:frakp_bound} and \eqref{eq:frakq_bound}, we have shown that
$$ \abs{ \Rk_{\alpha\beta}(z) - \bfR_{\alpha\beta}(z) } \lesssim n^{3\eps} \Psi^2 \sqrt{\sfT}. $$
Recall that $ \eta=n^{-2/3-\delta} $, $ \Psi = O(n^{-\frac{1}{3}+\delta+\eps}) $, and $ \sfT = O(n^{\frac{2}{3}-\epsilon_0}) $. Then we obtain
\begin{equation}\label{eq:Resolvent_Fixed_z}
n \eta \abs{ \Rk_{\alpha\beta}(z) - \bfR_{\alpha\beta}(z) } \leq n^{-\frac{\epsilon_0}{2} + \delta +5 \eps}.
\end{equation}
Choosing $ \delta + 5\eps< \frac{\epsilon_0}{2} $ yields the desired bound \eqref{eq:Resolvent_Stability} for a fixed $ z $.

So far, we have proved the desired result for a fixed $ z $. To extend this result to a uniform estimate, we simply invoke a standard net argument. To do this, we divide the interval $ [-n^{-2/3+\delta},n^{-2/3+\delta}] $ into $ n^2 $ sub-intervals, and consider $ z=E+\ii\eta $ with $ \kappa(z) $ taking values in each sub-interval. Note that
$$ |\bfR_{\alpha\beta}(z_1) -\bfR_{\alpha\beta}(z_2)| \leq \frac{|z_1 - z_2|}{\min(\im(z_1),\im(z_2))^2}. $$
For $ z_1 $, $ z_2 $ associated with the same sub-interval, we have
$$ n \eta |\bfR_{\alpha\beta}(z_1) -\bfR_{\alpha\beta}(z_2)| \leq n \eta \frac{n^{-2/3+\delta} n^{-2}}{\eta^2} \leq n^{-1+2\delta}, $$
which is of lower order compared with the error bound in \eqref{eq:Resolvent_Fixed_z}. This shows that, up to a small multiplicative factor, the desired error bound \eqref{eq:Resolvent_Stability} holds uniformly in each sub-interval with overwhelming probability.
Finally, thanks to the overwhelming probability, a union bound over the $ n^2 $ sub-intervals yields the desired uniform estimate \eqref{eq:Resolvent_Stability} for all $ z=E+\ii\eta $ with $ |E-\lambda_+| \leq n^{-2/3+\delta} $ and $ \eta=n^{-2/3-\delta} $.

Using the same arguments, we can prove a similar bound for the $ \Rk_{ij} $ and $ \bfR_{ij} $ blocks. Hence, we have shown the desired results.
\end{proof}

\subsection{Stability of the top eigenvalue}

As a consequence of the stability of the resolvent, we also have the stability of the top eigenvalue. This stability of the eigenvalue will play a crucial rule for the resolvent approximation of eigenvector statistics in the next subsection.

\begin{lemma}\label{lem:Stability_Eigenvalue}
Assume $ k \leq n^{5/3-\epsilon_0} $ for some $ \epsilon_0>0 $. Let $ 0<\delta<\delta_0 $ with $ \delta_0 $ as in Lemma \ref{lem:Stability_Resolvent}. For any $ \eps>0 $, with overwhelming probability, we have
$$ \abs{ \lambda - \Lk } \leq n^{-\frac{2}{3}-\delta+\eps}. $$
\end{lemma}

\begin{proof}
Without loss of generality, we assume that $ \lambda>\Lk $. Set $ \eta=n^{-2/3-\delta} $. By the spectral representation of the resolvent \eqref{eq:Resolvent_Spectral}, we have
$$ \im \bfR_{\alpha\alpha}(z) = \eta \sum_{\ell=1}^p \frac{|\v_\ell(\alpha)|^2 }{(\lambda_\ell - E)^2 + \eta^2} \geq  \frac{ \eta |\v(\alpha)|^2}{(\lambda-E)^2 + \eta^2} \geq \frac{\eta |\v(\alpha)|^2}{2\pth{\max(|\lambda-E|,\eta)}^2}. $$
By the pigeonhole principle, we know that there exists $ \alpha $ such that $ |\v(\alpha)| \geq p^{-1/2} $. Choosing this $ \alpha $ and $ z=\lambda+\ii\eta $, we obtain
\begin{equation}\label{eq:Im_R_lower}
p \eta^{-1} \im \bfR_{\alpha\alpha}(\lambda+\ii\eta) \geq \frac{1}{2\eta^2}.
\end{equation}

On the other hand, using the spectral representation of resolvent again for $ \Rk $, we have
$$ p \eta^{-1} \im \Rk_{\beta\beta}(z) = \sum_{m=1}^p \frac{p |\vk_m(\beta)|^2}{\big(\Lk_m - \lambda \big)^2+\eta^2}. $$
Pick $ \omega >0 $, we decompose the summation into two parts
$$ J_1 = \sum_{m=1}^{n^\omega} \frac{p |\vk_m(\beta)|^2}{\big(\Lk_m - \lambda \big)^2+\eta^2},\ \ J_2 = \sum_{m=n^\omega + 1}^p \frac{p |\vk_m(\beta)|^2}{\big(\Lk_m - \lambda \big)^2+\eta^2}. $$
Using delocalization of eigenvectors, for any $ \eps>0 $, with overwhelming probability, we have
\begin{equation}\label{eq:J1_old}
J_1 \lesssim \frac{n^{\omega + \eps}}{(\min_{1 \leq m \leq p} |\Lk_m - \lambda|)^2}.
\end{equation}
By the Tracy-Widom limit of the top eigenvalue (Lemma \ref{lem:Tracy-Widom}), for any $ \eps>0 $, with overwhelming probability, we have $ |\lambda - \lambda_+| \leq n^{-2/3+\eps} $. Also, as discussed in \eqref{eq:Eigenvalue_Diff}, the rigidity of eigenvalues yields that for all $ m \geq n^\omega $, with overwhelming probability,
$$ \lambda - \Lk_m \gtrsim m^{2/3} p^{-2/3}. $$
Then using delocalization again, with overwhelming probability, we have
\begin{equation}\label{eq:J2_old}
J_2 \leq \sum_{m=n^\omega + 1}^p \frac{n^\eps}{(\Lk_m - \lambda)^2} \lesssim n^\eps (n^\omega)^{-1/3} n^{4/3}.
\end{equation}
Again, since $ |\Lk-\lambda| \leq 2n^{-2/3+\eps} $, by choosing $ \omega=2\eps $ we have $ J_2 \leq J_1 $. Therefore, by \eqref{eq:J1_old} and \eqref{eq:J2_old}, we have shown that with overwhelming probability
$$ p \eta^{-1} \im \Rk_{\beta\beta}(\lambda+\ii\eta) \lesssim n^{3\eps} \pth{\min_{1 \leq m \leq p} |\Lk_m - \lambda| }^{-2}. $$
Note that the minimum is attained by $ \Lk $. This shows that
$$ n \eta^{-1} \im \Rk_{\alpha\alpha}(\lambda+\ii\eta)  \lesssim  n^{3\eps} |\Lk - \lambda|^{-2}. $$
Using Lemma \ref{lem:Stability_Resolvent} and \eqref{eq:Im_R_lower}, we have
\begin{equation*}
n\eta^{-1} \im \Rk_{\alpha\alpha}(\lambda+\ii\eta)
\geq n \eta^{-1} \pth{ \im \bfR_{\alpha\alpha}(\lambda+\ii\eta) - \abs{ \im \Rk_{\alpha\alpha}(\lambda+\ii\eta) - \im \bfR_{\alpha\alpha}(\lambda+\ii\eta)  } }
\geq \frac{1}{2\eta^2} - \frac{1}{n^{c}\eta^2} \gtrsim \frac{1}{\eta^2}.
\end{equation*}
Therefore, we have shown that, with overwhelming probability,
$$ \frac{1}{\eta^2} \lesssim n^{3\eps} \frac{1}{|\lambda - \Lk|^2}. $$
Recall $ \eta = n^{-2/3-\delta} $, and we conclude that
$$ \abs{\lambda - \Lk} \leq n^{-2/3 - \delta + 3 \eps}, $$
which proves the desired result thanks to the arbitrariness of $ \eps>0 $.
\end{proof}

\subsection{Proof of Theorem \ref{thm:main2}}

The final ingredient to prove the resampling stability is the following approximation lemma, which asserts that the product of entries in the eigenvector can be well approximated by the resolvent entries.

\begin{lemma}\label{lem:Resolvent_Approx}
Assume that $ k \ll n^{5/3-\epsilon_0} $ for some $ \epsilon_0>0 $. Let $ 0 < \delta <\delta_0 $ be as in Lemma \ref{lem:Stability_Resolvent}. Then, for $ z_0=\lambda+\ii \eta $ with $ \eta=n^{-2/3-\delta} $, there exists $ c'>0 $ such that with probability $ 1-o(1) $ we have
$$ \max_{\alpha,\beta} \abs{ \eta \im \bfR_{\alpha \beta}(z_0) - \v(\alpha) \v(\beta) } \leq n^{-1-c'},\ \ \mbox{and}\ \ \max_{\alpha,\beta} \abs{ \eta \im \Rk_{\alpha \beta}(z_0) - \vk(\alpha) \vk(\beta) } \leq n^{-1-c'}. $$
Similarly, we also have
$$ \max_{i,j} \abs{ \eta \im \frac{\bfR_{ij}(z_0)}{z_0} - \u(i) \u(j) } \leq n^{-1-c'},\ \ \mbox{and}\ \ \max_{i,j} \abs{ \eta \im \frac{\Rk_{ij}(z_0)}{z_0} - \uk(i) \uk(j) } \leq n^{-1-c'}. $$
\end{lemma}

\begin{proof}
For any $ \eps>0 $, we consider a general $ z=E+\ii\eta $ with $ |E-\lambda_+| \leq n^{-2/3+\eps} $. From the spectral representation of the resolvent \eqref{eq:Resolvent_Spectral}, we have
$$ \im \bfR_{\alpha\beta}(z) = \eta \sum_{\ell=1}^p \frac{\v_\ell(\alpha) \v_\ell(\beta)}{(\lambda_\ell - E)^2 + \eta^2}. $$
Pick some $ \omega>0 $, we decompose the summation on the right-hand side into three parts
$$ \sum_{\ell=1}^p \frac{\v_\ell(\alpha) \v_\ell(\beta)}{(\lambda_\ell - E)^2 + \eta^2} = \frac{\v(\alpha) \v(\beta)}{(\lambda-E)^2 + \eta^2} + J_1 + J_2, $$
where
$$ J_1 = \sum_{\ell=2}^{n^\omega} \frac{\v_\ell(\alpha) \v_\ell(\beta)}{(\lambda_\ell - E)^2 + \eta^2},\ \ \  J_2 = \sum_{\ell=n^\omega + 1}^p \frac{\v_\ell(\alpha) \v_\ell(\beta)}{(\lambda_\ell - E)^2 + \eta^2}. $$
Using the same arguments as in \eqref{eq:J2_old}, for any $ \eps>0 $, with overwhelming probability we have
$$ |J_2| \lesssim n^\eps (n^\omega)^{-1/3} n^{1/3}. $$
For the term $ J_1 $, we consider the following event
$$ \calE := \sth{\lambda_1 -\lambda_2 \geq c_0 n^{-2/3}} \cap \sth{\max_{1 \leq \ell \leq p} \|\v_\ell\|_\infty \leq n^{-1/2+\eps} } \cap \sth{|J_2| \lesssim n^\eps (n^\omega)^{-1/3} n^{4/3}}. $$
For any $ \eps>0 $, we can find an appropriate $ c_0>0 $ such that $ \P(\calE) > 1-\eps/2 $. Then, for $ z=E+\ii\eta $ with $ |\lambda-E| \leq \frac{c_0}{2}n^{-2/3} $, on the event $ \calE $, we have
$$ |J_1| \lesssim n^\eps n^\omega n^{1/3}. $$
Let $ \delta'>0 $ with $ \delta'+\delta<\delta_0 $. On the event $ \calE $, for all $ z=E+\ii\eta $ with $ |\lambda -E| \leq \eta n^{-\delta'} $ and $ \eta = n^{-2/3 - \delta} $, we have
$$ \abs{\v(\alpha)\v(\beta) - \frac{\eta^2 \v(\alpha) \v(\beta)}{(\lambda - E)^2 + \eta^2}} \leq n^{-1+2\eps} \abs{ 1 - \frac{\eta^2}{(\lambda-E)^2 + \eta^2} } \leq n^{-1+2\eps - 2\delta'}. $$
This yields
\begin{align*}
\abs{ \eta \im \bfR_{\alpha\beta}(z) - \v(\alpha) \v(\beta) } &\leq \abs{\v(\alpha)\v(\beta) - \frac{\eta^2 \v(\alpha) \v(\beta)}{(\lambda - E)^2 + \eta^2}} + \eta^2 (|J_1|+|J_2|)\\
&\leq n^{-1+2\eps - 2\delta'} + n^{-1+\eps+\omega-2\delta} + n^{-1+\eps-\frac{\omega}{3} - 2\delta}.
\end{align*}
Choosing $ \omega=\eps < \min(\delta,\delta')/2 $, we obtain
\begin{equation}\label{eq:R_Approx}
\max_{\alpha,\beta} \abs{ \eta \im \bfR_{\alpha\beta}(z) - \v(\alpha) \v(\beta) } \leq n^{-1-\min(\delta,\delta')}.
\end{equation}

Similarly, we can apply the same arguments to $ \Rk $. Consider the event
$$ \calE' := \sth{ \max_{\alpha,\beta} \abs{ \eta \im \Rk_{\alpha\beta}(z) - \vk(\alpha) \vk(\beta) } \leq n^{-1-\min(\delta,\delta')} \ \mbox{for all}\  |\Lk-E| \leq \eta n^{-\delta'} , \eta = n^{-2/3 -\delta} }. $$
By previous arguments, we know $ \P(\calE') > 1-\eps/2 $. This gives us $ \P(\calE \cap \calE') > 1- \eps $. Finally, note that $ \delta+\delta'<\delta_0 $, by Lemma \ref{lem:Stability_Eigenvalue}, with overwhelming probability we have $ |\lambda - \Lk|  \leq n^{-2/3 - \delta - \delta'} = \eta n^{-\delta'} $. This implies that we can choose $ z=\lambda + \ii\eta $ in both \eqref{eq:R_Approx} and $ \calE' $. Thus, we have shown the desired result for $ \v $ and $ \vk $ by choosing $ 0<c'<\min(\delta,\delta') $.

On the other hand, from \eqref{eq:Resolvent_Spectral} we also have
$$ \im \frac{\bfR_{ij}(z)}{z} = \eta \sum_{\ell=1}^n \frac{\u_\ell(i) \u_\ell(j)}{(\lambda_\ell - E)^2 + \eta^2}. $$
Using the same methods as above yields the desired result for $ \u $ and $ \uk $.
\end{proof}

Now we prove the main result Theorem \ref{thm:main2} on the stability of PCA under moderate resampling.

\begin{proof}[Proof of Theorem \ref{thm:main2}]
Let $ z_0=\lambda+\ii\eta $ as in Lemma \ref{lem:Resolvent_Approx}. By Lemma \ref{lem:Stability_Resolvent} and \ref{lem:Resolvent_Approx}, we know that, with probability $ 1-o(1) $, for all $ \alpha,\beta \in \calI_2 $, we have
\begin{align*}
\Big|\v(\alpha)\v(\beta) &- \vk(\alpha)\vk(\beta)\Big|\\
&\leq \abs{ \v(\alpha)\v(\beta) - \eta \im \bfR_{\alpha\beta}(z_0) } + \abs{ \eta \im \bfR_{\alpha\beta}(z_0) - \eta \im \Rk_{\alpha\beta}(z_0) } + \abs{ \eta \im \Rk_{\alpha\beta}(z_0)- \vk(\alpha)\vk(\beta)}\\
&\leq n^{-1-c} + n^{-1-c'} + n^{-1-c}.
\end{align*}
Denote $ c'':=\min(c,c') $, and we have
$$ \max_{\alpha,\beta} \abs{\v(\alpha)\v(\beta) - \vk(\alpha)\vk(\beta)} \lesssim n^{-1-c''}. $$
For any fixed $ \eps>0 $, we consider the event
$$ \calE := \sth{\max_{\alpha,\beta} \abs{\v(\alpha)\v(\beta) - \vk(\alpha)\vk(\beta)} \lesssim n^{-1-c''}} \cap \sth{\|\vk\|_\infty \leq n^{-1/2+\eps}}. $$
Since delocalization of eigenvectors holds with overwhelming probability, we know that $ \P(\calE)=1-o(1) $.

By the pigeonhole principle, there exists $ \alpha $ such that $ |\v(\alpha)|> p^{-1/2} $. We choose the $ \pm $ phases of $ \v $ and $ \vk $ in the way that $ \v(\alpha) $ and $ \vk(\alpha) $ are non-negative. On the event $ \calE $, we obtain
$$ \abs{\v(\alpha) - \vk(\alpha)} = \frac{\abs{(\v(\alpha))^2 - (\vk(\alpha))^2}}{\v(\alpha) + \vk(\alpha)} \lesssim n^{-1/2-c''}. $$
Moreover, for any entry $ \v(\beta) $ and $ \vk(\beta) $, if $ \calE $ holds, the triangle inequality gives us
\begin{align*}
\abs{\v(\beta) - \vk(\beta)} &= \frac{\abs{\v(\alpha)\v(\beta) - \v(\alpha)\vk(\beta)}}{\v(\alpha)}\\
& \leq \frac{\abs{\v(\alpha)\v(\beta) - \vk(\alpha)\vk(\beta)}}{\v(\alpha)} + \frac{|\vk(\beta)|}{\v(\alpha)} |\v(\alpha) - \vk(\alpha)|\\
&\lesssim n^{-1/2 - c''} + n^{-1/2 - c'' +\eps}.
\end{align*}
Choosing $ \eps $ sufficiently small, this implies the desired result \eqref{e.main2}. 

For $ \u $ and $ \uk $, note that
\begin{multline*}
\abs{\u(i)\u(j) - \uk(i)\uk(j)}\\
\leq \abs{ \u(i)\u(j) - \eta \im \frac{\bfR_{ij}(z_0)}{z_0} } + \abs{ \eta \im \frac{\bfR_{ij}(z_0)}{z_0} - \eta \im \frac{\Rk_{ij}(z_0)}{z_0} } + \abs{ \eta \im \frac{\Rk_{ij}(z_0)}{z_0}- \uk(i)\uk(j)}    
\end{multline*}
By Lemma \ref{lem:Stability_Resolvent}, we have
$$  \abs{ \im \frac{\bfR_{ij}(z_0)}{z_0} - \im \frac{\Rk_{ij}(z_0)}{z_0} } \leq \abs{ \frac{\bfR_{ij}(z_0) - \Rk_{ij}(z_0)}{z_0} } \lesssim \abs{ \bfR_{ij}(z_0) - \Rk_{ij}(z_0) } \leq \frac{1}{n^{1+c}\eta}.  $$
As a consequence, we have
$$ \abs{\u(i)\u(j) - \uk(i)\uk(j)} \lesssim n^{-1-c''}. $$
The desired result for $ \u $ and $ \uk $ then follows from the same arguments above for $ \v $ and $ \vk $.
\end{proof}

\section*{Acknowledgment}
The author thanks Yihong Wu and Paul Bourgade for helpful discussions at the early stage of this project. Thanks also to Yiyun He for helpful discussion on differential privacy.

\bibliography{Sensitivity}
\bibliographystyle{alpha}
\end{document}